\documentclass[pdflatex,sn-apa]{sn-jnl}

\usepackage{graphicx}
\usepackage{amsmath,amssymb,amsfonts}
\usepackage{amsthm}
\usepackage{mathtools}
\usepackage{booktabs}
\usepackage{xcolor}
\usepackage{url}
\usepackage{hyperref}  
\usepackage{appendix}

\newcommand{\R}{\mathbb{R}}
\newcommand{\E}{\mathbb{E}}
\newcommand{\tr}{\operatorname{tr}}
\newcommand{\col}{\operatorname{col}}
\newcommand{\Var}{\operatorname{Var}}
\newcommand{\limp}{\lim_{p\to\infty}}
\newcommand{\Pib}{\Pi}
\newcommand{\Pibp}{\Pi^{\perp}}

\newcommand{\PP}{\mathbb{P}}

\newcommand{\triplebar}{\vert\kern-0.25ex\vert\kern-0.25ex\vert}

\DeclarePairedDelimiterX{\inner}[2]{\langle}{\rangle}{#1, #2}

\newtheorem{theorem}{Theorem}
\newtheorem{lemma}{Lemma}
\newtheorem{proposition}{Proposition}
\newtheorem{corollary}{Corollary}
\newtheorem{assumption}{Assumption}
\theoremstyle{remark}
\newtheorem{remark}{Remark}

\newcommand{\pn}{{(p,n)}}
\newcommand{\nmax}{n_{\mathrm{max}}}
\newcommand{\pmax}{p_{\mathrm{max}}}

\begin{document}

\title[Principal component error in high-dimensional factor models]{Principal component error in high-dimensional factor models}

\author[1]{\fnm{Alex} \sur{Bernstein}}
\author[2]{\fnm{Lisa R.} \sur{Goldberg}}\email{lrg@berkeley.edu}
\author[3]{\fnm{Nicholas} \sur{Gunther}}
\author*[4]{\fnm{Alec} \sur{Kercheval}}\email{akercheval@fsu.edu}
\author[5]{\fnm{Tian} \sur{Lan}}
\author[6]{\fnm{Yian} \sur{Lin}}\email{ethanl1954@berkeley.edu}
\author[7]{\fnm{Darwin} \sur{Yao}}

\affil[1]{\orgdiv{} 
\orgaddress{\city{Berkeley}, \state{California}, \country{USA}}}
\affil[2]{\orgdiv{Department of Mathematics}, \orgname{University of California, Berkeley, and Blackrock, Inc.}, \orgaddress{\city{Berkeley and San Francisco}, \state{California}, \country{USA}}}
\affil[3]{\orgdiv{} 
\orgaddress{\city{Davis}, \state{California}, \country{USA}}}
\affil*[4]{\orgdiv{Department of Mathematics}, \orgname{Florida State University}, \orgaddress{\city{Tallahassee}, \state{FL}, \country{USA}}} 
\affil[5,6]{\orgdiv{} \orgname{University of California, Berkeley}, \orgaddress{ \country{USA}}}
\affil[7]{\orgdiv{} \orgname{University of Washington}, \orgaddress{\city{Seattle}, \country{USA}}}

\abstract{%
In a statistical factor model, principal components (or eigenvectors) of a sample covariance matrix serve as estimates of {\it principal directions}, the true drivers of co-movement of a collection of observed variables.  
We write the often substantial error in these estimates as a sum of
two interpretable terms, which we show have almost sure asymptotic limits as the number of variables grows with sample size bounded. This scenario is commonplace in financial economics, genomics, machine learning and signal processing. {\it Out-of-subspace error} measures the distance from an estimate to the subspace spanned by  population factor exposures. It can be expressed in terms of data,  providing an estimable floor for asymptotic error.  {\it In-subspace error} is driven by a {\it rotation error} within that subspace, which arises from
 the fixed sample size of the latent
 factor returns and cannot be estimated from data alone. We illustrate our error analysis with a three-factor simulation of the US public equity market, showing the dependence of the magnitude of the error and its components on dimension and sample size.  In that simulation, out-of-subspace error dominates.  Researchers who rely on principal component analysis to estimate factor models can use our results to quantify errors in model-based predictions and attributions.
}

\keywords{statistical factor model, principal direction, principal component analysis, eigenvector, estimation error, high dimension}

\maketitle

\section{Introduction}

Factor models are ubiquitous in modern quantitative finance and empirical economics, providing a parsimonious framework to explain the cross-section of returns, manage portfolio risk, and forecast economic outcomes. In the classical setup, a $p$-vector of observable variables $y$ follows a linear structure $$y = Bf + z.$$  The  product $Bf$ of the factor exposure {(or loading)}
matrix $B$ and the factor return vector $f$ is the systematic component of $y$, while the idiosyncratic (specific return) vector $z$ is the component of $y$ not explained by factors.   We assume a latent model, in which only $y$ is observed.  The central object of interest is often the space spanned by the columns of $B$ or, more specifically, its {\it principal directions}: unit-length linear combinations of the exposures $B$ that maximize explained variance. They serve as fundamental building blocks for constructing factor-mimicking portfolios, estimating risk premia, and performing asset pricing tests.

The setting of this article is the high-dimension, low-sample-size (HL) asymptotic regime, where parameters are numerous and data are scarce.
This is commonplace in finance, since  regime changes often leave only a limited time series of relevant data available for estimation of model parameters.
In this article, we analyze how well  true principal directions are estimated by leading eigenvectors of a sample covariance matrix in the HL asymptotic regime for data generated by a factor model. This type of estimation is standard practice in principal component analysis.

\paragraph{The main result.}
Denote by $b_j$ the $j$th strongest principal direction, and by $h_j$
the $j$th eigenvector of the sample covariance
matrix $S$ computed from $n$ observations of $p$ variables. Our first result, Theorem~\ref{thm:main}, asserts that when $n$ is fixed and $p\to\infty$,
\begin{equation*}
    \sin^2 \angle (h_j, b_j) \xrightarrow[]{} \operatorname{OE}_j + \operatorname{IE}_j,
    \qquad
    \operatorname{IE}_j:=(1- \operatorname{OE}_j) \operatorname{RE}_j ,
\end{equation*}
almost surely. Here, the \emph{out-of-subspace error} $\operatorname{OE}_j$ is the limiting squared sine of the angle between $h_j$ and the span $\mathcal{B}$ of the $b_j$s (the ``subspace'' in question, equivalently, $\mathcal{B} = \col(B)$);   the \emph{rotation error} $\operatorname{RE}_j$ is the limiting squared sine of the angle between $b_j$ and the projection of $h_j$ onto $\mathcal{B}$, an angle of rotation within $\mathcal B$; and the \emph{in-subspace error} $\operatorname{IE}_j$ is the rotation error weighted by the squared cosine $1-\operatorname{OE}_j$.  We use these names for the limits throughout, and write $\operatorname{OE}_j^{(p)}$ and $\operatorname{RE}_j^{(p)}$ for the corresponding finite-$p$ quantities.  The out-of-subspace error is explicit:
\begin{equation*}
  \operatorname{OE}_j=\frac{\delta^2}{n\lambda_j+\delta^2}=\frac{1}{1+\mathrm{SNR}_j},
  \qquad
  \mathrm{SNR}_j:=\frac{n\lambda_j}{\delta^2},
\end{equation*}
where $\delta^2$ is the asymptotic average specific (idiosyncratic) variance and $\lambda_j$ is the asymptotic realized variance of the $j$th systematic component over the $n$ observations, both suitably scaled (Sections~\ref{sec:data} and~\ref{sec:gram}).
When data are scarce, we naturally expect estimation error to be large.  Perhaps unexpectedly, in the HL regime these errors stabilize almost surely for large $p$.

\paragraph{Two sources of error.}
The two terms have different causes.  Out-of-subspace error is caused by the specific returns: it vanishes when $\delta^2=0$, and it equals the share of the $j$th sample eigenvalue that is contributed by noise.  Rotation error is caused by the finite sample of factor returns: in the limit it is exactly the error that principal component analysis would make if it were applied to the noise-free systematic data $BF$ (Remark~\ref{rem:noisefree}), and it vanishes when the $j$th eigenvector of the realized covariance of the (suitably normalized) factor returns coincides with its population counterpart.  The mechanism behind this clean separation is simple.  As $p$ grows, the specific returns add an isotropic term to the $n\times n$ matrix of inner products between observations, which shifts its eigenvalues without moving its eigenvectors (see \eqref{eq:Wlim} in Section~\ref{sec:gram}). 

\paragraph{Estimable and non-estimable error.}
We complement Theorem~\ref{thm:main} with two further theorems. Theorem~\ref{thm:obsfloor} shows that $\operatorname{OE}_j$ is consistently estimated by $\ell/\theta_j$, the ratio of the average non-zero bulk eigenvalue to the $j$th leading eigenvalue of the sample covariance matrix.  This ratio is computable from data and serves as an asymptotic lower bound for the total error.  In contrast, $\operatorname{RE}_j$, and with it $\operatorname{IE}_j$, is not determined by data alone; this is the content of Theorem~\ref{thm:rotrange}.

Our results bear on questions faced by a practitioner fitting a latent factor model to data:
\begin{itemize}
\item  For a given number of observations, how large are the errors in estimated principal directions?
 \item How many observations are required to provide reasonable confidence that  estimated principal directions are accurate enough for a particular application?
\item  How accurately have we identified the span of the factor exposures $B$?
\end{itemize}
Data alone answer the third question (Corollary~\ref{cor:subspace}) and provide a lower bound for the answer to the first.  Complete answers to the first two questions require, in addition, a model for the law of the factor returns, because the rotation component is not determined by the data; Section~\ref{subsec:usage} describes what such a model must supply, and Section~\ref{sec:simulation} illustrates the answers in a calibrated simulation.
Answers to these questions can inform, for example, the choice of estimation window: a short window is more plausibly drawn from a single regime, while a longer window reduces estimation error but may call for a model that accommodates changes in factor volatility.

The remainder of the paper is organized as follows. Section~\ref{sec:context} gives a brief overview of the literature and summarizes our contributions; detailed comparisons with the closest antecedents are deferred to Section~\ref{sec:prior}, after our results are stated.   Section \ref{sec:fac-mod} presents the factor model and defines the estimation targets. Section \ref{sec:data} describes the data, scaling, and assumptions. Section~\ref{sec:gram} reduces the problem to matrices of fixed size. Theorem~\ref{thm:main} is stated and proved in Section~\ref{sec:main-thm}. Section \ref{sec:data-driven} shows that the out-of-subspace error is {\it estimable}, meaning that it can be estimated from data and its estimator has the same asymptotic limit as the error, while the rotation error is {\it non-estimable}, in the sense that it is not determined by the data (without further distributional assumptions).
 Section \ref{sec:simulation} presents simulation results illustrating the convergence properties of the errors.  The technical convergence theory supporting the proofs is developed in Appendix~\ref{sec:lemmas}, while the proof of non-estimability of the rotation error appears in
 Appendix~\ref{sec:orbit}.
 Table~\ref{tab:thm1} in Section~\ref{sec:main-thm} lists the quantities that appear in our main results, and Table~\ref{tab:matrices} in Appendix~\ref{sec:tables} summarizes the full notation.

 \section{Related literature and new contributions}\label{sec:context}

\subsection{Factor and spiked models}

Statistical factor models and their close cousins, spiked models (or spiked covariance models), make high dimensional data tractable.  Modern results on both types of models are asymptotic in the number of variables, $p$.

  A spiked model is often specified in terms of the spectral structure of a population covariance matrix, by separating a few larger (spiked) eigenvalues from smaller (bulk) eigenvalues. The number of spiked eigenvalues stays fixed as $p$ grows. In some settings, such as \cite{johnstone2001}, \cite{paul2007} and \cite{Johnstone2009}, all eigenvalues stay bounded as $p$ grows.  In other settings the spiked eigenvalues grow like $p^\alpha$ with $\alpha>0$ while the bulk eigenvalues stay bounded \citep{jung2009, jung2012, yata2012, shen2016, wang2017}; \cite{johnstone2018} give an orientation that contrasts the two. We focus on the latter case, with $\alpha=1$, since that setting corresponds to empirical models in financial economics and other disciplines.

In contrast to a spiked model, a factor model is specified in terms of a data generating process that explicitly separates variables into factor and specific components. 
  When the factors are pervasive in the sense  that their exposures are not too concentrated (Assumption~\ref{asm:gram}), linear growth with $p$ of the spike eigenvalues is a consequence. 
Factor models are ubiquitous in financial economics.  Important references include \cite{sharpe1963}, \cite{rosenberg1974}, \cite{ross1976}, \cite{Fama1992} and \cite{Fama1993}.  Our factor models are statistical (or principal component analysis (PCA)) models, but there are other types, as summarized in \cite{connor1995}.

Since a factor model gives rise to a version of a spiked model, 
we freely use terminology from both the spiked model and factor model literature, factor return and signal, specific return and noise, throughout this paper.

\subsection{Asymptotic regimes for estimation}

Much of the modern literature on factor and spiked models works in regimes in which the number of variables $p$ and the number of observations $n$ tend to infinity together.  In the high dimension, high sample size (HH) regime of random matrix theory, $p$ and $n$ grow in proportion; this framework has roots in \cite{wigner1955}, \cite{wigner1958} and \cite{marcenko1967}, and its consequences for PCA are surveyed in \cite{johnstone2018}.  The econometric literature on large factor models, notably \cite{bai2002} and \cite{bai2003}, surveyed in \cite{bai2008}, lets $p$ and $n$ diverge jointly subject only to mild rate restrictions such as $\sqrt{p}/n\to0$, without a proportional limit, and relies on laws of large numbers rather than on random matrix theory.  A third, intermediate setting, in which $n\to\infty$ but $p/n\to\infty$ is permitted, is treated in \cite{wang2017} and in Section~5.1 of \cite{shen2016}.

In this paper, we use techniques from the high dimension low sample size asymptotic regime (HL, also called HDLSS), where the dimension $p$ tends to infinity but the sample size $n$ is bounded  \citep{hall2005, ahn2007, jung2009, shen_shen_marron2016, shen2016, jung2018, aoshima2018}.    
This relatively recent approach models settings where the number of variables is large, but the number of available or relevant observations is limited. 
This situation comes about, for example,  in financial markets with a large number of assets where only a short data history is relevant to covariance forecasts.
As we show, there are insights from the HL regime that do not emerge in the HH regime.  

Comparisons of the regimes are discussed in \cite{shen2016}, \cite{aoshima2018}, \cite{johnstone2018} and \cite{goldberg2023}.

\subsection{Antecedents to the results in this article}

Early  work in financial economics considers factor models of returns of securities  whose number grows without bound, and with parameter estimates based on a finite time series.
\cite{ross1976} developed the arbitrage pricing theory and  \cite{chamberlain1983} introduced the concept of approximate factor models in finance. \cite{connor1986}, \cite{connor1988} and \cite{connor1993} developed {\it asymptotic principal components}, weighted combinations of right singular vectors of data matrices, to generate consistent estimators of the space spanned by the factor returns. While these works predate the formal introduction of the HL regime, they pioneer some of its technical tools.  An example is {\it Gram reduction}, which exploits the fact that for a $p \times n$ data matrix $Y$, $Y Y^\top/pn$ and $Y^\top Y/pn$  have the same non-zero eigenvalues. With $n$ fixed and $p$ tending to $\infty$, we can take limits in the latter to learn about the former.  We use this technique extensively in this article.

A series of articles on James--Stein shrinkage correction of sample eigenvectors serve as precursors to the present article. \cite{goldberg2022, shkolnik2022, gurdogan2022, goldberg2023, goldberg2020, goldberg2025, yoon2025} include analysis of the leading sample eigenvector for the case of a one-factor model, with various stationarity and independence assumptions.  Explorations of the multi-factor case are in \cite{goldberg2020}, \cite{gurdogan2024} and \cite{shkolnik2025}.  These articles are motivated by the well-known problem that estimation error in a covariance matrix leads mean-variance optimizers to underforecast the risk of optimized portfolios \citep{michaud1989, elkaroui2010, elkaroui2013}, a problem that \cite{bianchi2017} document for latent factor models in particular.  They show that while the leading sample eigenvector is not a consistent estimator of the population counterpart, there is a consistent estimator of the eigenvector {\it error}.  This estimator emerges from a strong law of large numbers 
under bounded fourth moments, and allows for James--Stein type corrections to estimated eigenvectors in the HL regime. While it is not stated explicitly in the references,  James-Stein shrinkage in HL is related to the phenomenon of concentration of measure \citep{talagrand1995}. This line of thinking implicitly underlies some of the phenomena analyzed in this article.

The closest antecedents of the present article are \cite{jung2012}, \cite{jung2022} and \cite{shen2016}, who study spiked HL models in which the spikes grow linearly in the dimension, which is the relevant case for factor models; \cite{gurdogan2024}, who treat the multi-factor case in a related formulation; and \cite{wang2017}, who obtain large-$n$ counterparts of our error terms.  We compare our results with these in detail in Section~\ref{sec:prior}, after they are stated.

\subsection{Our contributions}

 The present article provides an HL analysis of principal direction estimation error in a latent factor model. Factor models are typically $\alpha=1$ spiked models, and the  additional factor structure supplies interpretability, separating the cross-sectional noise level $\delta^2$ from the time series of factor returns, and supports an analysis conditional on the realized path of factor returns.
Our results apply even when factor returns are serially dependent and whose marginal laws may vary over time, subject to a common second moment, and when specific returns are
cross-sectionally correlated with uniformly bounded specific covariance eigenvalues (Assumption~\ref{asm:noise} and
Appendix~\ref{subsec:crosscorr}).
  They are organized into three theorems, described next.

Theorem~\ref{thm:main} decomposes error in an estimated principal direction into interpretable pieces and provides almost sure limits in the number of variables $p$ for each, when sample size $n$ is fixed.   There are two essential, limiting components.  The first is out-of-subspace error $\operatorname{OE}$, which measures the distance of a principal component to the subspace spanned by the factor exposures.  A spiked model analog of our result, {when there is no factor structure,} is \citet[Theorem 5]{shen2016}. The second component is rotation error $\operatorname{RE}$, which
reflects
the deviation of the {fixed-$n$} sample factor covariance matrix from its population counterpart, and which is the limiting error of principal component analysis applied to noise-free data (Remark~\ref{rem:noisefree}).
Rotation error is implicit in the joint limits of \citet[Theorem 2]{jung2012}, appears as the rotation part of the score bias in \cite{jung2022}, and is the reason \cite{shen2016} measure angles only to the spiked subspace.  We are unaware of a treatment that isolates it as a term with a path-conditional almost sure limit, or that addresses its estimability.

Theorem~\ref{thm:obsfloor} provides a data-driven estimate of out-of-subspace error $\operatorname{OE}$ in terms of eigenvalues of a sample covariance matrix.
Our easy-to-compute estimate has the same almost sure limit as the out-of-subspace error itself.  It yields a computable floor for principal component error in a statistical factor model.
With more than one factor, there is no longer a consistent estimator of the angles between the
sample and population eigenvectors.  This is the content of Theorem~\ref{thm:rotrange} and Corollary~\ref{cor:noest}:
the rotation error $\operatorname{RE}$ is not determined by the data (holding the data fixed, it can take any value in $[0,1]$), and consequently no consistent estimator of it exists.

\section{The factor model and the estimation target} \label{sec:fac-mod}

A random $p$-vector $y = y^{(p)}$ follows a $k$-factor linear model
\begin{equation}\label{eq:lfm}
  y =Bf+z,
\end{equation}
with $p\gg k$, deterministic loadings $B=B^{(p)}\in\R^{p\times k}$ of rank
 $k$, a random  $k$-vector of factor returns $f$ (that has no dependence on $p$), and a random specific (or idiosyncratic) $p$-vector $z=z^{(p)}$.  Only $y$ is observed.  We take $f$ and
$z$ to have mean zero and finite second moments, with $z$ uncorrelated with
$f$. Write $\Sigma_f=\E[ff^{\!\top}]\in\R^{k\times k}$ (positive definite),
$\Delta = \Delta^{(p)} =\E[zz^{\!\top}]\in\R^{p\times p}$, and $\Sigma_y = \E[yy^\top]$.  Then
\begin{equation}\label{eq:cov}
  \Sigma_y=B\Sigma_f B^{\!\top}+\Delta=\Sigma_0+\Delta,
  \quad \text{ where }\,
  \Sigma_0:= \operatorname{Cov}(Bf) = B\Sigma_f B^{\!\top}\ \ (\text{rank }k).
\end{equation}

Because $B$ and $f$ enter \eqref{eq:lfm} only through their product $Bf$, and because that product does not have a unique factorization, the components $B$ and $f$ are not separately identifiable.
Instead we study the eigenstructure of the {\em scaled systematic covariance matrix}  $\Sigma_0/p$.
Assuming its $k$ nonzero eigenvalues are distinct (a result of Assumption \ref{asm:sep} and Proposition \ref{prop:chart}(ii) for all sufficiently large $p$), the spectral decomposition theorem gives
\begin{equation}\label{eq:sdt}
  \Sigma_0=b\,\Delta_0\,b^{\!\top},
\end{equation}
where $\Delta_0/p = \operatorname{diag}(\mu_1^{(p)},\dots,\mu_k^{(p)})$ is the $k \times k$ diagonal matrix of the $k$ leading
eigenvalues of $\Sigma_0/p$ in decreasing order and
$b=[\,b_1\ \cdots\ b_k\,]\in\R^{p\times k}$ has
orthonormal unit eigenvectors as columns.  The $b_j$s are the
 {\it principal directions} of the systematic covariance matrix, unique up to sign, and they are the targets of estimation, with
\begin{equation}\label{eq:borth}
  b^{\!\top}b=I_k,\qquad \col(b)=\col(B)=:\mathcal B.
\end{equation}
  We can rewrite  factor model (\ref{eq:lfm}) as
\begin{align}\label{eq:lfmpd}
y = b \phi + z.
\end{align}
This is worked out in Lemma \ref{lem:basis} of Appendix~\ref{sec:lemmas}, which provides the linear transformation of $B$ to $b$ and $f$ to $\phi$.
The $k$-vector $\phi$ of transformed factor returns
depends on $p$ while the original $k$-vector $f$ does not.  Lemma~\ref{lem:basis} also shows the covariance matrix of $\phi$ is
  the diagonal matrix
  \begin{equation}
       \Sigma_\phi := \E[\phi \phi^\top] = \Delta_0.
  \end{equation}
 Thus the systematic contribution $b\phi=\sum_{j}\phi_j b_j$ decomposes along the orthonormal
principal directions $b_j$ with uncorrelated coefficients $\phi_j$ of variance
$p\mu_j^{(p)} = (\Delta_0)_{jj}$, and we
can write
\[
\Sigma_y = b \Sigma_\phi b^\top + \Delta .
\]

 Estimates of principal directions $b_j$ and their associated variances $p\mu_j$ are the building blocks of financial risk models,  used for portfolio construction and return attribution by quantitative investors. These estimates are typically eigenvectors and eigenvalues of sample covariance matrices generated by principal component analysis.  We quantify the accuracy of these estimates in the results that follow.
\section{Model and Data Assumptions  }\label{sec:data}

We fix the number $k \geq 1$ of factors and consider a sequence of models \eqref{eq:lfm} indexed by an increasing dimension $p$.
This sequence  supports our asymptotic analysis and formalizes the intuitive process of adding more and more variables to an existing pool of variables
without changing the ones previously added.
{ Therefore if convenient we may think of the loading matrix sequence $B^{(p)}$ as nested in the sense that $B^{(p)}$ is obtained as the first $p$ rows of a fixed infinite array $\{B_{ij}: i \geq 1, \, 1 \leq j \leq k \}$. However, this nesting property is not required for our results as long as the sequence $B^{(p)}$ satisfies Assumption~\ref{asm:gram} below.}

 Up to this point, all quantities are {\em population} quantities that may depend on $p$ but not on any observations. We now introduce {\em sample} or {\em realized} quantities based on a finite number of observations.  Suppose we have $n$ observations
$y^{(p,l)}, l = 1,\dots,n$, determined
by $n$ samples $(f^{(l)},z^{(p,l)}) \in \mathbb{R}^{k}\times\mathbb{R}^{p}$, $l=1,\dots,n$, according to the model \eqref{eq:lfm}. Each sample pair satisfies the baseline assumptions of
Section~\ref{sec:fac-mod}: mean zero, and $z^{(l)}$ uncorrelated with
$f^{(l)}$, but, subject to the requirements of Assumptions~\ref{asm:fm}--\ref{asm:delta} below,
the samples need not be
independent, nor identically distributed across samples, because the factor
returns are only weakly restricted; the specific returns, by contrast, may be cross-sectionally correlated, but are
stationary and independent across samples, conditional on the factor path
(Assumption~\ref{asm:noise}).  

The observations $y^{(p,1)},\dots,y^{(p,n)}$ are collected as the columns of a
$p\times n$ data matrix,
\begin{equation}\label{eq:data}
  Y=Y^{(p)}=b\,\Phi+Z,
\end{equation}
where $\Phi=\Phi^{(p)}\in\R^{k\times n}$ holds the $n$ realizations of
$\phi$  (derived from the realized factor path
$F = [f^{(1)},\dots,f^{(n)}] \in\R^{k\times n}$ via Lemma~\ref{lem:basis}) and $Z=Z^{(p)} = [z^{(p,1)},\dots, z^{(p,n)}] \in\R^{p\times n}$ the realizations of
$z$. Since the sample size $n$ is fixed throughout most of this article,  we suppress dependence on $n$ in most of our notation.  Exceptions occur 
in Subsection~\ref{subsec:varyn} and are clearly marked.

{The dependence of these models on $p$ is constrained by the numbered assumptions below. In particular, we assume the noise cross-sections are nested:} for each observation $l$, $1 \leq l \leq n$, there is a single
infinite sequence of specific returns $z_1^{(l)},z_2^{(l)},\dots$, and $z^{(p,l)}$
consists of its first $p$ elements.  Equivalently, the noise realizations
form one infinite array $\{Z_{il}: i\ge1,\ 1\le l\le n\}$ from which the
$p\times n$ noise matrix $Z^{(p)}$ at cross-section size $p$ is obtained; growing $p$
adds coordinates without changing the law
of those already present. Assumption~\ref{asm:noise} below requires in addition that, conditional on the factor path, the noise is stationary across observations.

The asymptotics are taken in $p$ with $n>k$ fixed, conditional on a
realization $F \in \R^{k\times n}$ that does not change with $p$.
The estimators are  leading sample eigenvalues and eigenvectors   of the
\emph{scaled sample covariance matrix}
\begin{equation}\label{eq:S}
S =   S^{(p)}:=\frac{YY^{\!\top}}{np}\in\R^{p\times p},
\end{equation}
    whose $j$th unit eigenvector we denote $h_j=h_j^{(p)}$ (eigenvalue
$\theta_j^{(p)}$), for $j=1,\dots,k$, in decreasing order   (the $1/p$
scaling keeps the leading eigenvalues
bounded, via the assumptions below).
We compare $h_j$ with its corresponding principal direction
$b_j$ through $\sin^2\!\angle(h_j,b_j)$.

Notation: To avoid any angle ambiguity due to sign flips, we adopt the definition
\begin{equation}\label{eq:angledef}
\angle(u,v):=\arccos\frac{|\langle u,v\rangle|}{\|u\|\,\|v\|}
   \;\in\;[0,\tfrac{\pi}{2}]
\end{equation}
for any nonzero $u,v\in\R^m$,
denoting the acute angle between the {lines} spanned by $u$ and $v$.  This is
the intrinsic notion for eigenvectors, which are determined only up to
sign.
More generally, for a nonzero $u\in\R^m$ and a subspace $\mathcal V\subseteq\R^m$ with
orthogonal projector $\Pi_{\mathcal V}$, we   write
\begin{equation}\label{eq:angledefsub}
  \angle(u,\mathcal V):=\arccos\frac{\|\Pi_{\mathcal V}u\|}{\|u\|}
   \;\in\;[0,\tfrac{\pi}{2}]
\end{equation}
for the angle between $u$ and its closest direction in $\mathcal V$.

\paragraph{Overview of the assumptions.}
The six assumptions below have distinct roles.  Assumption~\ref{asm:fm} fixes the second moments of the factor returns, which define the targets $b_j$, and places no other restriction on the factor path.  Assumptions~\ref{asm:noise} and~\ref{asm:delta} control the specific returns.  Their role in the proofs is limited: the specific returns enter Theorems~\ref{thm:main}, \ref{thm:obsfloor}
and~\ref{thm:rotrange} and Corollary~\ref{cor:subspace} only through two consequences, both established in Proposition~\ref{prop:dual}:
\[
  \text{(P1)}\ \ \frac{Z^{\!\top}Z}{p}\longrightarrow\delta^2 I_n ,
  \qquad
  \text{(P2)}\ \ \|b^{\!\top}Z\|_{\mathrm{Fr}}=o(\sqrt p).
\]
In words, (P1) says that, after scaling, the specific return vectors at different dates are asymptotically orthogonal and of equal length; (P2) says that they have asymptotically negligible projection onto the factor subspace.  Assumptions~\ref{asm:gram} and~\ref{asm:sep} are population conditions on the systematic component, and Assumption~\ref{asm:reg} is a genericity condition on the realized factor path $F$.

\medskip

\begin{assumption}[Factor returns]\label{asm:fm}
All random objects are defined on a common probability space
$(\Omega,\mathcal F,\mathbb P)$.  The factor path
$F=(f^{(1)},\dots,f^{(n)})$ is a random element of $\R^{k\times n}$
satisfying, for each $l=1,\dots,n$,
\[
  \E\|f^{(l)}\|^2<\infty,\qquad \E f^{(l)}=0,\qquad
  \E\big[f^{(l)}f^{(l)\top}\big]=\Sigma_f,
\]
with $\Sigma_f\in\R^{k\times k}$ positive definite and not depending on $l$.
No further restriction is imposed on the law of $F$: the marginal laws of
the $f^{(l)}$ may differ, and the cross-moments
$\E[f^{(l)}f^{(m)\top}]$, $l\ne m$, are unrestricted, so the factors may be
serially dependent and conditionally heteroskedastic.
\end{assumption}
\medskip

\begin{remark}
Assumption~\ref{asm:fm} constrains only the second moments of the marginals
of $F$; it is what makes $\Sigma_0=B\Sigma_fB^{\!\top}$, and hence the
targets $b_1,\dots,b_k$ of \eqref{eq:sdt}, well defined.  It is used
nowhere in the asymptotics, which are conditional on the realized path.

\end{remark}

\medskip

\begin{assumption}[Specific returns]\label{asm:noise}
The specific returns array $\{Z_{il}: i\ge1,\ 1\le l\le n\}$ is generated from
the factor path $F=(f^{(1)},\dots,f^{(n)})$ by a probability kernel $Q$: for
every $F\in\R^{k\times n}$, $Q(F,\cdot)$ is a probability measure on
$\R^{\mathbb N\times n}$; for every measurable $A\subseteq\R^{\mathbb N\times n}$,
the map $F\mapsto Q(F,A)$ is measurable; and the joint law of the factor path
and the array is given by
$\PP(F\in C,\ \{Z_{il}\}\in A)=\int_C Q(F,A)\,\PP(F\in dF)$.
Thus $Q(F,\cdot)$ is the conditional law of the array given the path $F$, and
the phrase ``conditional on $F$'' and the notations $\E[\,\cdot\mid F]$,
$\operatorname{Cov}(\,\cdot\mid F)$ and $\Var(\,\cdot\mid F)$ refer to it.
For every $F\in\R^{k\times n}$, under $Q(F,\cdot)$ the columns
$z^{(1)},\dots,z^{(n)}$ of the array are mutually independent, mean zero, and
identically distributed (stationary across observations).  For each $p$,
the conditional covariance matrix
$\Delta^{(p)}:=\operatorname{Cov}\big(z^{(p,l)}\mid F\big)\in\R^{p\times p}$
does not depend on $F$; write $\delta_i^2:=\Delta^{(p)}_{ii}$ for the
specific variances.  There is a constant $C<\infty$ such that, for all
$p$, $l$ and $F$,
\begin{align*}
  \text{(i)}&\quad \sup_{a\in\R^p,\ \|a\|=1}
     \E\big[(a^{\!\top}z^{(p,l)})^4\mid F\big]\le C,\\
    \text{(ii)}&\quad \Var\big(\|z^{(p,l)}\|^2\mid F\big)\le C\,p .
\end{align*}
\end{assumption}

\medskip

\begin{remark}
By the nesting convention, $\Delta^{(p)}$ is the leading $p\times p$ block
of $\Delta^{(p')}$ whenever $p'\ge p$, so $\delta_i^2$ does not depend on
$p\ge i$.  For a unit vector $a$, Jensen's inequality and condition~(i)
give $(a^{\!\top}\Delta^{(p)}a)^2\le\E[(a^{\!\top}z^{(p,l)})^4\mid F]\le C$;
taking $a$ to be a unit eigenvector,
\begin{equation}\label{eq:Deltabd}
  \|\Delta^{(p)}\|\le\sqrt C\quad\text{for all }p,
  \qquad\text{and in particular}\qquad
  \sup_i\delta_i^2\le\sqrt C .
\end{equation}
Thus Assumption~\ref{asm:noise} does not require independence across
assets: it imposes the approximate factor structure of
\cite{chamberlain1983}, in which the specific covariance has uniformly
bounded eigenvalues.  Condition~(ii) ensures that $\|z^{(p,l)}\|^2$ deviates from its mean
$\tr\Delta^{(p)}\sim\delta^2p$ by $O_P(\sqrt p)$; it holds, for example, when
$\sup_i\sum_j|\operatorname{Cov}(Z_{il}^2,Z_{jl}^2\mid F)|\le C$.
Condition~(i) controls the correlation of the specific returns themselves.

Because Assumption~\ref{asm:noise} conditions on the whole factor path, it
forces $\E[z^{(l)}f^{(m)\top}]=0$ for every pair $l,m$, and it makes the
observation covariance $\Sigma_y^{(l)}=\Sigma_0+\Delta^{(p)}$ the same for
every $l$, so that the targets $b_j$ are common to all observations
(Lemma~\ref{lem:uncorr} of the Appendix).

The assumption specifies the conditional law $Q(F,\cdot)$ at every path $F$,
not merely at almost every path under a particular law of $F$.  This is what
gives meaning to statements about an individual path, and it allows the same
noise mechanism $Q$ to be combined with different laws of the factor path, as
in Section~\ref{subsec:nonestimable}.
\end{remark}

\medskip

\begin{assumption}[Average specific variance]\label{asm:delta}
$\dfrac1p\sum_{i=1}^p\delta_i^2=\dfrac{\tr\Delta^{(p)}}{p}\longrightarrow\delta^2>0$ as
$p\to\infty$.
\end{assumption}
\begin{remark}
The limiting average specific variance $\delta^2$ plays an important role in the error estimates that we develop below.
\end{remark}
\medskip

\begin{assumption}[Loading Gram convergence]\label{asm:gram}
$\dfrac{B^{\!\top}B}{p}\longrightarrow G_B$ as $p \to \infty$, for some positive definite
$G_B\in\R^{k\times k}$.
\end{assumption}
Write $G_B^{(p)}:=B^{\!\top}B/p$ for the finite-$p$ loading Gram matrix, so that $G_B^{(p)}\to G_B$.
\medskip

\begin{remark}When the entries of $B$ are uniformly bounded, as in typical financial applications, Assumption~\ref{asm:gram} 
implies that the loadings on each factor are
{\em pervasive} in the sense that a positive fraction of the loadings are bounded away from zero.  
\end{remark}
\medskip

\begin{assumption}[Population separation]\label{asm:sep}
The deterministic symmetric matrix $M_f := G_B^{1/2}\,\Sigma_f\,G_B^{1/2}\in\R^{k\times k}$ has
distinct eigenvalues $\mu_1>\cdots>\mu_k$.
\end{assumption}

Here and throughout, the exponent $1/2$ denotes the symmetric positive definite square root.  The $\mu_j$ are positive, since $G_B$ and $\Sigma_f$ are positive definite, and they are the limiting leading
eigenvalues of $\Sigma_0/p$, $\mu_j^{(p)}\to\mu_j$; see
Proposition~\ref{prop:chart}(ii).
\medskip

\begin{assumption}[Regular event]\label{asm:reg}
The realized $k\times k$ symmetric matrix $M_F:=G_B^{1/2}\,\Sigma_F\,G_B^{1/2}$, where $\Sigma_F:=FF^{\!\top}/n$, has distinct positive
eigenvalues $\lambda_1>\cdots>\lambda_k>0$.
\end{assumption}

By Lemma~\ref{lem:gramdual}(i), applied to $A=G_B^{1/2}F/\sqrt n$ (so that $AA^{\!\top}=M_F$ and $A^{\!\top}A=F^{\!\top}G_BF/n$), Assumption~\ref{asm:reg} is equivalent to requiring that the $n\times n$ matrix $F^{\!\top}G_BF/n$ have $k$ distinct positive eigenvalues, which are then the $\lambda_j$; either form forces $\operatorname{rank}F=k$.  Assumptions~\ref{asm:sep} and~\ref{asm:reg} are population and realized versions of the same simple-spectrum condition on the transformed factor returns $G_B^{1/2}f$, whose population and sample second-moment matrices are $M_f$ and $M_F$ (Section~\ref{sec:gram}).

According to Assumptions \ref{asm:fm} and \ref{asm:noise}, the samples $y^{(l)}$ may be serially dependent through the factors, and need not be identically distributed. Beyond its conditional mean and covariance, the conditional law of the specific returns may vary with the whole factor path $F$, subject to conditions~(i) and~(ii), but by stationarity not with the observation $l$; by the nesting convention preceding \eqref{eq:data}, the law of coordinates already present does not vary with $p$. In particular,
a volatility regime that differs across the $n$ dates is admissible through the factor returns, but not through the specific returns, whose covariance $\Delta^{(p)}$ is the same at every date.

 Our first three assumptions aim to balance realism and parsimony in the specification of our factor model.  We allow serially correlated and heteroskedastic factor returns $f^{(l)}$ (Assumption~\ref{asm:fm}), two empirically verified features of daily factor returns in US equity markets. As long as $f$ has a fixed unconditional covariance structure $\Sigma_f$ (Assumption~\ref{asm:fm}) and the noise $z$ is stationary across observations with cross-sectional average variance $\delta^2$ (Assumptions~\ref{asm:noise} and \ref{asm:delta}),
we can consider the data to be representative of a single regime. 

Assumption~\ref{asm:noise} is formulated in terms of moments, rather
than through a generating mechanism, because the specific returns enter
the proofs only through (P1) and (P2).
Conditions~(i) and~(ii) are what Lemma~\ref{lem:zconc} needs to deliver
(P1) and (P2) almost surely.  In Appendix~\ref{subsec:crosscorr}
(Proposition~\ref{prop:suff}) we verify them under three primitive
conditions: conditionally independent entries with uniformly bounded
fourth moments, for which $\Delta^{(p)}$ is diagonal; Gaussian specific
returns with $\sup_p\|\Delta^{(p)}\|<\infty$; and linear representations
in independent innovations with $\sup_p\|\Delta^{(p)}\|<\infty$.  The last
admits, for example, spatially correlated specific returns driven by
heavy-tailed (such as Student~$t$) innovations.  In each case the moment
requirement is of fourth order.  Cross-sectional correlation of this
strength does not change the limits in our results, which depend on the
specific returns only through $\delta^2$.  Appendix~\ref{subsec:crosscorr}
also discusses what the assumption excludes (a common date-specific volatility
shock, which violates~(ii) and~(P1)), specific covariances
whose largest eigenvalue grows with $p$, and how
stationarity can be relaxed.

Assumptions~\ref{asm:sep} and~\ref{asm:reg} are generic conditions allowing us to line up the corresponding limiting eigenvalues, as described in the proof. Assumption~\ref{asm:reg} holds almost surely under any joint law of the path $F$ with a density on $\R^{k\times n}$ (recall $n>k$).  

\section{Gram reduction} \label{sec:gram}

Given a factorization $AA^\top$ of a positive semidefinite matrix, we call
the matrix $A^\top A$ a {\em dual Gram matrix} of $AA^\top$; it is the matrix of inner products of the columns of $A$.
Since a positive semidefinite matrix has many such factorizations, a dual Gram depends on the choice of $A$.
If $A$ has more rows than columns, the dual Gram $A^\top A$ is smaller than $AA^\top$. We then call the passage from $AA^\top$ to $A^\top A$ {\em Gram reduction}.
In this section we define certain dual Gram matrices in terms of particular factorizations that we specify.

To study the $p$-asymptotic limit, we handle the growing dimension of the $p \times p$ scaled sample covariance matrix $S^{(p)}$ by passing to a fixed size $n \times n$ dual Gram $W^{(p)}$ with limit $W$ (see below), and its  systematic component $W_0$. While this allows us to study the sample asymptotics, it is not sufficient for analyzing estimation error because the population principal directions $b_1,b_2,\dots,b_k$ are unrelated to $n$. To overcome this obstacle, we make a further Gram reduction of the rank $k$ matrix $W_0$ to a $k \times k$ matrix  $M_F$, the sample covariance of suitably transformed factor returns. Its eigenvectors are realized systematic directions that can be compared directly with those of the corresponding population matrix $M_f$.  Table~\ref{tab:thm1} in Section~\ref{sec:main-thm} lists the quantities that appear in our main results, and Table~\ref{tab:matrices} summarizes the full notation.

\paragraph{The $n\times n$ dual $W^{(p)}$.}  
The matrices $S^{(p)} = YY^{\!\top}/(np) \in \mathbb{R}^{p \times p}$ and
\begin{equation}\label{eq:Wdef}
  W^{(p)}:=\frac{Y^{\!\top}Y}{np}\in\R^{n\times n}
\end{equation}
share their nonzero eigenvalues (Lemma~\ref{lem:gramdual}); $W^{(p)}$ has
the same fixed size $n$ at every $p$.  Its $(l,m)$ entry is the scaled inner
product $\langle Y_{\cdot l},Y_{\cdot m}\rangle/(np)$ of two observations.
Substituting $Y=BF+Z$,
 this converges (Proposition~\ref{prop:dual}), as $p \to \infty$, to
\begin{equation}\label{eq:Wlim}
  W:=\underbrace{\frac{F^{\!\top}G_BF}{n}}_{=:W_0}
       +\frac{\delta^2}{n}\,I_n ,
\end{equation}
a systematic dual Gram $W_0$ plus an \emph{isotropic} noise term: 
noise that is independent, or only weakly dependent, across
the $p$ coordinates averages to the same variance $\delta^2/n$ at each
observation.  The noiseless part $W_0$ is the dual Gram of $M_F$ via $A=G_B^{1/2}F/\sqrt n$, so by Lemma~\ref{lem:gramdual}(i) and Assumption~\ref{asm:reg} it has rank $k$ and simple nonzero eigenvalues $\lambda_1>\cdots>\lambda_k>0$; write its
$j$th eigenpair as $(\lambda_j,w_j)$.  Adding $(\delta^2/n)I_n$ shifts every eigenvalue
by $\delta^2/n$ and leaves unchanged the eigenvectors, so the top $k$ eigenpairs of
$W$ are $(\lambda_j+\delta^2/n,\,w_j)$, separated from the
remaining eigenvalue $\delta^2/n$ (multiplicity $n-k$) by the distance (eigenvalue gap)
$\lambda_k>0$.

\paragraph{The $k\times k$ matrices $M_F$ and $M_f$.}
Recall
\begin{equation}\label{eq:Mdef}
  M_F=G_B^{1/2}\,\Sigma_F\,G_B^{1/2},
  \qquad
  M_f=G_B^{1/2}\,\Sigma_f\,G_B^{1/2},
  \qquad
  \Sigma_F:=\frac{FF^{\!\top}}{n},
\end{equation}
the matrices of Assumptions~\ref{asm:reg} and~\ref{asm:sep} respectively.  They are the sample and population second-moment matrices of the transformed factor returns $G_B^{1/2}f$.  The first, $M_F$, is a Gram reduction of $W_0$ as previously noted.
Both $M_F$ and $W_0$ have the simple positive eigenvalues $\lambda_1>\cdots>\lambda_k>0$ of Assumption~\ref{asm:reg}. We denote by $\omega_j$ the $j$th unit eigenvector of $M_F$, so its $j$th eigenpair is $(\lambda_j, \omega_j)$.  The eigenvector correspondence of Lemma~\ref{lem:gramdual} gives the duality link
\begin{equation}\label{eq:omegadual}
  G_B^{1/2}F\,w_j=\sqrt{n\lambda_j}\;\omega_j,
\end{equation}
with the sign of $\omega_j$ matched to that of $w_j$. The second, $M_f$,  has distinct positive eigenvalues $\mu_1>\cdots>\mu_k>0$ by Assumption~\ref{asm:sep}. Denote by $\beta_j$ its $j$th unit eigenvector, so its $j$th eigenpair is $(\mu_j, \beta_j)$.

The realized direction $\omega_j$ and the population direction $\beta_j$ are thus the $j$th eigenvectors of the sample and population second-moment matrices, $M_F$ and $M_f$, of the same $k$-vector $G_B^{1/2}f$, and their angle $\angle(\omega_j,\beta_j)$ is responsible for the in-subspace error below.  Only $\beta_j$ involves $\Sigma_f$; only $\omega_j$ involves the realized path $F$.  The angle is a finite-$n$ phenomenon: if $\Sigma_F\to\Sigma_f$ as $n\to\infty$, then $M_F\to M_f$ and $\angle(\omega_j,\beta_j)\to0$ (Proposition~\ref{prop:n-limit}).

In the special case $G_B = I_k$, where the scaled loadings $B/\sqrt p$ are asymptotically orthonormal, $\angle(\omega_j,\beta_j)$ is exactly the angle between corresponding eigenvectors of the sample and population factor covariance matrices $\Sigma_F$ and $\Sigma_f$.  In general the factor returns are first transformed by $G_B^{1/2}$, which accounts for the unequal lengths of, and the overlaps among, the columns of $B$.

To apply these ideas in a typical practical setting, consider a researcher who has a $p \times n$ matrix of $n$ returns of $p$ securities and no prior idea about factor exposures.  That researcher may begin their analysis by extracting spiked eigenvectors and  eigenvalues from the scaled sample covariance matrix $S = Y Y^\top/pn $, while taking equation (\ref{eq:lfmpd}),
\[
y = b \phi + z,
\]
as the starting point, where $b \in \mathbb{R}^{p \times k}$ has orthonormal columns, $\phi \in \mathbb{R}^k$ has mean zero and diagonal covariance matrix, and $z \in \mathbb{R}^p$ is as before.

To apply our asymptotic theory, we can consider this single-$p$ model as embedded in a sequence of models
\[
y^{(p)} = b^{(p)} \phi^{(p)} + z^{(p)}
\]
generating data
\[
Y^{(p)} = b^{(p)} \Phi^{(p)} + Z^{(p)}.
\]
Without affecting the initial single-$p$ model, we are free to assume the new sequence of models satisfies
\begin{enumerate}
    \item the same nested property for $Z^{(p)}$ as before, and
    \item $f^{(l)} := \phi^{(p,l)}/\sqrt{p}$, and hence $F = (f^{(1)}, \dots, f^{(n)})$, is independent of $p$.
\end{enumerate}
In this case, letting $B = \sqrt{p}\; b$ gives the equivalent formulation
\[
y = Bf + z
\]
as in our earlier treatment. Since $B^\top B = p\, b^\top b  = p I_k$ and hence $G_B = I_k$, Assumption~\ref{asm:gram} is satisfied. Under this and our other assumptions, our results follow.
In this situation our matrices $M_F$ and $M_f$ reduce to
\[
M_F = FF^\top/n = \Phi \Phi^\top /(np)
\]
and the diagonal matrix
\[
M_f = \Sigma_f = \E[\phi \phi^\top]/p.
\]
The angles between corresponding eigenvectors $\omega_j$ and $\beta_j$ of $M_F$ and $M_f$ that appear in Theorem~\ref{thm:main} now reduce to the angles between eigenvectors of $FF^\top/n = \Phi \Phi^\top /(np)$ and the corresponding standard basis vectors of $\mathbb{R}^k$.

\section{Errors in principal directions}  \label{sec:main-thm}

The $j$th principal component $h_j \in \mathbb{R}^p$ is a sample quantity that is an
imperfect estimator of the population principal direction $b_j$. We measure the difference
by the angle $\angle_p (h_j, b_j)$, or, equivalently,
$\sin^2 \angle_p (h_j, b_j)$.
The subscript in $\angle_p$ emphasizes that the angle is formed
in $\mathbb{R}^p$ for each $p$; an undecorated $\angle$ refers to angles in the fixed spaces $\mathbb{R}^n$ or $\mathbb{R}^k$.
Theorem~\ref{thm:main} below decomposes
the $p$-asymptotic limit of the fixed-sample-size error
into terms that can be further analyzed.  The proof appears in Section~\ref{sec:main-proof}, making use of results collected in Appendix~\ref{sec:lemmas}. Section~\ref{sec:data-driven} takes up analysis of the terms.

It may come as a surprise that the sample error has a $p$-asymptotic limit for fixed sample size $n$.
In high dimensions, independent noise vectors are
nearly orthogonal and nearly equal in length after scaling, a phenomenon that is explained by concentration of measure; this is property (P1) of Section~\ref{sec:data}.


Throughout, statements of the form ``conditional on $F$ and almost surely
as $p\to\infty$'' are to be read as follows: for every path
$F\in\R^{k\times n}$ satisfying Assumption~\ref{asm:reg}, the stated
convergence holds with probability one under $Q(F,\cdot)$, the
conditional law of the noise array $\{Z_{il}\}$ given $F$ specified in
Assumption~\ref{asm:noise}. That assumption requires that the conditional mean and covariance of each column $z^{(l)}$ not depend on the path $F$, but permits the rest of its conditional distribution, such as higher moments, to vary with the realization of $F$ (though, by stationarity, not with the observation index $l$). Our path-conditional statements upgrade automatically to unconditional ones by the tower property, as long as $F$ satisfies Assumption~\ref{asm:reg} a.s.

\subsection{Error decomposition}\label{subsec:split}

Let $\Pib:=bb^{\!\top}$ be the orthogonal projector of $\R^p$ onto $\mathcal B=\col(b)$, and $\Pibp:=I-\Pib$. Projecting $h_j$ onto $\mathcal B$ and $\mathcal{B}^\perp$ yields an exact split of principal direction error at every $p$:
\begin{equation}\label{eq:thmsplit}
  \sin^2\!\angle_p(h_j,b_j)
  \;=\;
  \sin^2\!\angle_p(h_j,\mathcal B)
  \;+\;
  \cos^2\!\angle_p(h_j,\mathcal B)\;\sin^2\!\angle_p\big(\Pib h_j,\,b_j\big),
\end{equation}
where the second summand is read as $0$ if $\Pib h_j=0$.
Indeed, since $b_j\in\mathcal B$ we have $\langle h_j,b_j\rangle=\langle \Pib h_j,b_j\rangle$, so
$\cos^2\!\angle_p(h_j,b_j)=xy$ with $x:=\|\Pib h_j\|^2=\cos^2\!\angle_p(h_j,\mathcal B)$ and
$y:=\cos^2\!\angle_p(\Pib h_j,b_j)$; then use $1-xy=(1-x)+x(1-y)$.
We write
\[
  \operatorname{OE}_j^{(p)}:=\sin^2\!\angle_p(h_j,\mathcal B),
  \qquad
  \operatorname{RE}_j^{(p)}:=\sin^2\!\angle_p\big(\Pib h_j,\,b_j\big)
\]
for the finite-$p$ out-of-subspace and rotation errors, so that \eqref{eq:thmsplit} reads $\sin^2\!\angle_p(h_j,b_j)=\operatorname{OE}_j^{(p)}+(1-\operatorname{OE}_j^{(p)})\operatorname{RE}_j^{(p)}$.

Let $\lambda_j$, $\omega_j$ and $\beta_j$ be as in Section~\ref{sec:gram}: $(\lambda_j,\omega_j)$ is the $j$th eigenpair of $M_F$, and $\beta_j$ is the $j$th unit eigenvector of $M_f$, where $M_F=G_B^{1/2}(FF^{\!\top}/n)G_B^{1/2}$ and $M_f=G_B^{1/2}\Sigma_fG_B^{1/2}$.

\medskip

\begin{theorem}[Error decomposition]\label{thm:main}
Let Assumptions~\ref{asm:fm}--\ref{asm:reg} hold.  For  $j\in\{1,\dots,k\}$:

\medskip\noindent\emph{(i) Principal direction error converges for large $p$.} Conditional on $F$ and almost
surely as $p\to\infty$,
\begin{equation}\label{eq:thm}
  \sin^2\!\angle_p(h_j,b_j)
  \;\longrightarrow\;
  \underbrace{\frac{\delta^2}{n\lambda_j+\delta^2}}
            _{\text{out-of-subspace error }}
  \;+\;
  \underbrace{\frac{n\lambda_j}{n\lambda_j+\delta^2}\,
              \sin^2\!\angle(\omega_j,\,\beta_j)}
            _{\text{in-subspace error }}.
\end{equation}

\medskip\noindent\emph{(ii) Each component of principal direction error converges for large $p$.}  Conditional on
$F$ and almost surely as $p\to\infty$, one has $\Pib h_j\neq0$ for all
sufficiently large $p$, and
\begin{equation}\label{eq:thmfloor}
  \sin^2\!\angle_p(h_j,\mathcal B)\;\longrightarrow\;
    \frac{\delta^2}{n\lambda_j+\delta^2},
  \qquad
  \cos^2\!\angle_p(h_j,\mathcal B)\;\longrightarrow\;
    \frac{n\lambda_j}{n\lambda_j+\delta^2},
\end{equation}
\begin{equation}\label{eq:thmrot}
  \sin^2\!\angle_p\big(\Pib h_j,\,b_j\big)\;\longrightarrow\;
    \sin^2\!\angle\big(\omega_j,\,\beta_j\big) .
\end{equation}
\end{theorem}

By Theorem~\ref{thm:main}(ii), $\operatorname{OE}_j^{(p)}\to\operatorname{OE}_j:=\delta^2/(n\lambda_j+\delta^2)$ and $\operatorname{RE}_j^{(p)}\to\operatorname{RE}_j:=\sin^2\!\angle(\omega_j,\beta_j)$, and the in-subspace error in \eqref{eq:thm} is $\operatorname{IE}_j=(1-\operatorname{OE}_j)\operatorname{RE}_j$; these are the limits named in the Introduction.
As the specific variance $\delta^2$ tends to zero, the principal direction error becomes equal to the rotation error
$\operatorname{RE}_j=\sin^2\!\angle(\omega_j,\,\beta_j)$, which measures the angle between the realized
direction $\omega_j$ and the population direction $\beta_j$.  The next remark makes this precise for every $\delta^2$: in the limit, the rotation error is the error of principal component analysis applied to the noise-free data.  (Under the
assumptions of Proposition~\ref{prop:n-limit} of the Appendix, both terms of \eqref{eq:thm} vanish as
$n\to\infty$, since, by that Proposition,
$\lambda_j\to\mu_j$ and
$\angle(\omega_j,\beta_j)\to0$.)
\medskip

\begin{remark}[Rotation error is the error of noise-free PCA]\label{rem:noisefree}
Let $S_0^{(p)}:=BFF^{\!\top}B^{\!\top}/(np)$ be the \emph{noiseless sample covariance}, the scaled sample covariance of the systematic part $BF$ of the data.  Under the assumptions of Theorem~\ref{thm:main}, for all sufficiently large $p$ the $j$th eigenvalue $\lambda_j^{(p)}$ of $S_0^{(p)}$ is simple and positive, $\lambda_j^{(p)}\to\lambda_j$, and its unit eigenvector $h_j^0$ satisfies
\begin{equation}\label{eq:noisefree}
  \sin^2\!\angle_p\big(h_j^0,\,b_j\big)\;\longrightarrow\;\sin^2\!\angle\big(\omega_j,\,\beta_j\big)=\operatorname{RE}_j .
\end{equation}
This convergence involves no noise and holds for each fixed path $F$.  Moreover, conditional on $F$ and almost surely, $\angle_p(\Pib h_j,h_j^0)\to0$: within the factor subspace, the sample principal component $h_j$ asymptotically points along the noise-free principal component $h_j^0$.  Consequently, conditional on $F$ and almost surely,
\[
  \sin^2\!\angle_p(h_j,b_j)\;-\;\Big[\sin^2\!\angle_p(h_j,\mathcal B)+\cos^2\!\angle_p(h_j,\mathcal B)\,\sin^2\!\angle_p\big(h_j^0,b_j\big)\Big]\;\longrightarrow\;0 .
\]
The proof is given at the end of Section~\ref{sec:main-proof}.
\end{remark}

\begin{table}[htbp]
\centering
\small
\renewcommand{\arraystretch}{1.3}
\begin{tabular}{@{}lll@{}}
\toprule
Symbol & Meaning & Type\\
\midrule
$h_j$ & $j$th unit eigenvector of the sample covariance $S=YY^{\!\top}/(np)$ & sample\\
$b_j$ & $j$th unit eigenvector of the systematic covariance $\Sigma_0=B\Sigma_fB^{\!\top}$ & population\\
$\mathcal B$ & factor subspace $\col(B)=\operatorname{span}(b_1,\dots,b_k)$ & population\\
$\delta^2$ & limiting average specific variance (Assumption~\ref{asm:delta}) & population\\
$M_F$ & $G_B^{1/2}\Sigma_FG_B^{1/2}$, \ $\Sigma_F=FF^{\!\top}/n$ \eqref{eq:Mdef} & realized\\
$M_f$ & $G_B^{1/2}\Sigma_fG_B^{1/2}$ \eqref{eq:Mdef} & population\\
$\lambda_j,\ \omega_j$ & $j$th eigenvalue and unit eigenvector of $M_F$ & realized\\
$\mu_j,\ \beta_j$ & $j$th eigenvalue and unit eigenvector of $M_f$ & population\\
$h_j^0$ & $j$th unit eigenvector of $S_0^{(p)}=BFF^{\!\top}B^{\!\top}/(np)$ (Remark~\ref{rem:noisefree}) & realized\\
$\theta_j,\ \ell$ & $j$th eigenvalue and average bulk eigenvalue of $W^{(p)}=Y^{\!\top}Y/(np)$ & sample\\
$\operatorname{OE}_j,\ \operatorname{RE}_j,\ \operatorname{IE}_j$ & out-of-subspace, rotation and in-subspace errors (limits) & ---\\
\bottomrule
\end{tabular}
\caption{Quantities appearing in Theorems~\ref{thm:main}--\ref{thm:rotrange}.  ``Sample'' quantities are computable from the data $Y$; ``realized'' quantities depend on the unobserved factor path $F$.  The full notation is in Table~\ref{tab:matrices}.}
\label{tab:thm1}
\end{table}

\begin{remark}
Dividing top and bottom by $n$ shows that the out-of-subspace error is the fraction of scaled variance $\lambda_j + \delta^2/n$ attributable to noise $\delta^2/n$, while the in-subspace error is weighted by the fraction attributable to signal $\lambda_j$.
In terms of the finite-sample realized signal-to-noise ratio $\mathrm{SNR}_j := n\lambda_j/\delta^2$, Theorem~\ref{thm:main} reads
\begin{equation}\label{eq:thm_snr}
  \sin^2\!\angle_p(h_j,b_j)
  \;\longrightarrow\;
  \frac{1}{1+\mathrm{SNR}_j}
  \,+\,
  \frac{\mathrm{SNR}_j}{1+\mathrm{SNR}_j}\,
  \sin^2\!\angle(\omega_j,\beta_j).
\end{equation}
\end{remark}

\subsection{Proof of the error decomposition} \label{sec:main-proof}

\paragraph{Setup: the loading frame.}
Since $B$ has rank $k$, the $p\times k$ matrix
\begin{equation}\label{eq:bbar}
  \bar b:=B\,(B^{\!\top}B)^{-1/2}
\end{equation}
has orthonormal columns spanning $\mathcal B=\col(B)=\col(b)$, so $\bar b\bar b^{\!\top}=bb^{\!\top}$ is the orthogonal projector onto $\mathcal B$.  Unlike the principal frame $b$ of \eqref{eq:sdt}, the \emph{loading frame} $\bar b$ is determined by the loadings alone and does not involve $\Sigma_f$.
With $G_B^{(p)}=B^{\!\top}B/p$, we have $B=\sqrt p\,\bar b\,(G_B^{(p)})^{1/2}$.  Then the scaled systematic covariance and the noiseless sample covariance take the parallel forms
\begin{equation}\label{eq:frameforms}
  \frac{\Sigma_0}{p}=\bar b\,M_f^{(p)}\,\bar b^{\!\top},
  \qquad
  S_0^{(p)}=\bar b\,M_F^{(p)}\,\bar b^{\!\top},
\end{equation}
where
\begin{equation}\label{eq:Mpdef}
  M_f^{(p)}:=(G_B^{(p)})^{1/2}\,\Sigma_f\,(G_B^{(p)})^{1/2},
  \qquad
  M_F^{(p)}:=(G_B^{(p)})^{1/2}\,\Sigma_F\,(G_B^{(p)})^{1/2}
\end{equation}
are finite-$p$ versions of \eqref{eq:Mdef}, converging to $M_f$ and $M_F$ by Assumption~\ref{asm:gram}.  In particular, $M_f^{(p)}$ has the eigenvalues $\mu_j^{(p)}$ of $\Sigma_0/p$, and the estimation targets are
\begin{equation}\label{eq:bjframe}
  b_j=\bar b\,\beta_j^{(p)},\qquad
  \beta_j^{(p)}:=\bar b^{\!\top}b_j
  =\text{$j$th unit eigenvector of $M_f^{(p)}$}
\end{equation}
(Proposition~\ref{prop:chart}(iii)), with $\angle(\beta_j^{(p)},\beta_j)\to0$ (Proposition~\ref{prop:chart}(iv)).
The $n\times n$ dual of $M_F^{(p)}$, via $A=(G_B^{(p)})^{1/2}F/\sqrt n$, is
$W_0^{(p)}:=F^{\!\top}G_B^{(p)}F/n$, which is also the dual of $S_0^{(p)}$ via $A=BF/\sqrt{np}$, and $W_0^{(p)}\to W_0$ by Assumption~\ref{asm:gram}.

Fix $F$ as in Section~\ref{sec:data} and work on the event, of conditional probability one, on which the conclusions of Proposition~\ref{prop:dual} hold.  All limits, and the symbols $o(\cdot)$, $O(\cdot)$, refer to $p\to\infty$ with $n$ and $k$ fixed.  Fix $j\le k$.
Eigenvectors are determined only up to sign; every quantity below is sign-free, and eigenvector convergence is used only in the form $u^{(p)}u^{(p)\top}\to uu^{\!\top}$, which is equivalent to $\angle(u^{(p)},u)\to0$ by Lemma~\ref{lem:econv}(ii).
The inputs are, from Proposition~\ref{prop:dual}(a),(b),(d) and Proposition~\ref{prop:chart}(ii),(iv):
\begin{gather*}
  \frac{Z^{\!\top}Z}{p}\to\delta^2I_n,\qquad
  \|b^{\!\top}Z\|_{\mathrm{Fr}}=o(\sqrt p),\qquad
  \theta_j^{(p)}\to\lambda_j+\frac{\delta^2}{n},\\
  w_j^{(p)}w_j^{(p)\top}\to w_jw_j^{\!\top},\qquad
  (G_B^{(p)})^{1/2}\to G_B^{1/2},\qquad
  \beta_j^{(p)}\beta_j^{(p)\top}\to\beta_j\beta_j^{\!\top}.
\end{gather*}
Write $c_j^2:=n\lambda_j/(n\lambda_j+\delta^2)\in(0,1)$.

\paragraph{Step 1: reduction to the $n\times n$ dual.}
By Proposition~\ref{prop:dual}(c),(d) and Lemma~\ref{lem:econv}(i), there is a finite $p_0$, depending on the realization, such that for $p\ge p_0$ the top $k$ eigenvalues $\theta_1^{(p)}>\cdots>\theta_k^{(p)}>0$ of $W^{(p)}$ are simple; we consider only such $p$.  Lemma~\ref{lem:gramdual}(ii), applied to $A=Y/\sqrt{np}$, together with $Y=BF+Z$ and $B=\sqrt p\,\bar b\,(G_B^{(p)})^{1/2}$ (Proposition~\ref{prop:chart}(i)), gives, up to sign,
\begin{equation}\label{eq:recon}
  h_j=\frac{Y\,w_j^{(p)}}{\sqrt{np\,\theta_j^{(p)}}}
     =\bar b\,\xi_j^{(p)}+\frac{Z\,w_j^{(p)}}{\sqrt{np\,\theta_j^{(p)}}},
  \qquad
  \xi_j^{(p)}:=\frac{(G_B^{(p)})^{1/2}F\,w_j^{(p)}}{\sqrt{n\,\theta_j^{(p)}}}\in\R^k .
\end{equation}

\paragraph{Step 2: out-of-subspace error.}
Since $\Pibp\bar b=0$ and $\|\Pibp v\|^2=\|v\|^2-\|b^{\!\top}v\|^2$, \eqref{eq:recon} gives
\[
  \|\Pibp h_j\|^2=\frac{1}{n\,\theta_j^{(p)}}
  \Big(w_j^{(p)\top}\tfrac{Z^{\!\top}Z}{p}\,w_j^{(p)}-\tfrac{1}{p}\|b^{\!\top}Zw_j^{(p)}\|^2\Big).
\]
The first term in parentheses tends to $\delta^2$, because $Z^{\!\top}Z/p\to\delta^2I_n$ in spectral norm and $\|w_j^{(p)}\|=1$; the second is at most $\|b^{\!\top}Z\|_{\mathrm{Fr}}^2/p=o(1)$.  Hence
\begin{equation}\label{eq:floor}
  \|\Pibp h_j\|^2\longrightarrow\frac{\delta^2}{n\lambda_j+\delta^2},
  \qquad
  \|\Pib h_j\|^2=1-\|\Pibp h_j\|^2\longrightarrow c_j^2 .
\end{equation}

\paragraph{Step 3: alignment with $b_j$.}
By Proposition~\ref{prop:chart}(iii), $\bar b^{\!\top}b_j=\beta_j^{(p)}$, so \eqref{eq:recon} gives
\[
  \langle h_j,b_j\rangle=\langle\xi_j^{(p)},\beta_j^{(p)}\rangle+r_p,
  \qquad
  |r_p|=\frac{|b_j^{\!\top}Zw_j^{(p)}|}{\sqrt{np\,\theta_j^{(p)}}}
  \le\frac{\|b^{\!\top}Z\|_{\mathrm{Fr}}}{\sqrt{np\,\theta_j^{(p)}}}=o(1).
\]
By the inputs above and the duality link \eqref{eq:omegadual}, $G_B^{1/2}Fw_j=\sqrt{n\lambda_j}\,\omega_j$,
\[
  \xi_j^{(p)}\xi_j^{(p)\top}
  =\frac{(G_B^{(p)})^{1/2}F\,w_j^{(p)}w_j^{(p)\top}F^{\!\top}(G_B^{(p)})^{1/2}}{n\,\theta_j^{(p)}}
  \longrightarrow
  \frac{G_B^{1/2}F\,w_jw_j^{\!\top}F^{\!\top}G_B^{1/2}}{n\lambda_j+\delta^2}
  =c_j^2\,\omega_j\omega_j^{\!\top}.
\]
Taking the trace, $\|\xi_j^{(p)}\|^2\to c_j^2$, so $\xi_j^{(p)}$ is bounded; and, using $\beta_j^{(p)}\beta_j^{(p)\top}\to\beta_j\beta_j^{\!\top}$,
\[
  \langle\xi_j^{(p)},\beta_j^{(p)}\rangle^2
  =\tr\big(\xi_j^{(p)}\xi_j^{(p)\top}\beta_j^{(p)}\beta_j^{(p)\top}\big)
  \longrightarrow c_j^2\tr\big(\omega_j\omega_j^{\!\top}\beta_j\beta_j^{\!\top}\big)
  =c_j^2\langle\omega_j,\beta_j\rangle^2 .
\]
Since $\langle h_j,b_j\rangle^2-\langle\xi_j^{(p)},\beta_j^{(p)}\rangle^2=r_p\big(2\langle\xi_j^{(p)},\beta_j^{(p)}\rangle+r_p\big)\to0$,
\begin{equation}\label{eq:align}
  \langle h_j,b_j\rangle^2\longrightarrow c_j^2\cos^2\!\angle(\omega_j,\beta_j).
\end{equation}

\paragraph{Conclusion.}
For (i), $\sin^2\!\angle_p(h_j,b_j)=1-\langle h_j,b_j\rangle^2\to 1-c_j^2\cos^2\!\angle(\omega_j,\beta_j)=(1-c_j^2)+c_j^2\sin^2\!\angle(\omega_j,\beta_j)$, which is \eqref{eq:thm}.
For (ii), \eqref{eq:thmfloor} is \eqref{eq:floor}.  Since $c_j^2>0$, $\Pib h_j\ne0$ for large $p$, and since $\langle h_j,b_j\rangle=\langle\Pib h_j,b_j\rangle$,
\[
  \cos^2\!\angle_p(\Pib h_j,b_j)=\frac{\langle h_j,b_j\rangle^2}{\|\Pib h_j\|^2}\longrightarrow\cos^2\!\angle(\omega_j,\beta_j)
\]
by \eqref{eq:floor} and \eqref{eq:align}, which is \eqref{eq:thmrot}.
\hfill$\qed$

\paragraph{Proof of Remark~\ref{rem:noisefree}.}
We continue with the notation and the almost sure event of the proof of Theorem~\ref{thm:main}.
By \eqref{eq:frameforms}, $S_0^{(p)}=\bar b\,M_F^{(p)}\,\bar b^{\!\top}$ with $\bar b^{\!\top}\bar b=I_k$, so the nonzero eigenvalues of $S_0^{(p)}$ are those of $M_F^{(p)}$, with unit eigenvectors $\bar b\,\omega_j^{(p)}$, where $\omega_j^{(p)}$ is a unit eigenvector of $M_F^{(p)}$.  Since $M_F^{(p)}\to M_F$ and $M_F$ has the simple positive eigenvalues $\lambda_1>\cdots>\lambda_k>0$ (Assumption~\ref{asm:reg}), Lemma~\ref{lem:econv}(i) shows that for large $p$ the eigenvalue $\lambda_j^{(p)}$ is simple and positive, $\lambda_j^{(p)}\to\lambda_j$, and $\omega_j^{(p)}\omega_j^{(p)\top}\to\omega_j\omega_j^{\!\top}$.  Thus $h_j^0=\pm\bar b\,\omega_j^{(p)}$ and, by \eqref{eq:bjframe},
\[
  \langle h_j^0,b_j\rangle^2=\langle\omega_j^{(p)},\beta_j^{(p)}\rangle^2
  =\tr\big(\omega_j^{(p)}\omega_j^{(p)\top}\beta_j^{(p)}\beta_j^{(p)\top}\big)
  \longrightarrow\langle\omega_j,\beta_j\rangle^2 ,
\]
which gives \eqref{eq:noisefree}; no probability is involved.  Next, by \eqref{eq:recon},
\[
  \pm\langle h_j,h_j^0\rangle=\langle\xi_j^{(p)},\omega_j^{(p)}\rangle+r_p',
  \qquad
  |r_p'|=\frac{|\omega_j^{(p)\top}\bar b^{\!\top}Zw_j^{(p)}|}{\sqrt{np\,\theta_j^{(p)}}}
  \le\frac{\|\bar b^{\!\top}Z\|_{\mathrm{Fr}}}{\sqrt{np\,\theta_j^{(p)}}}=o(1),
\]
since $\|\bar b^{\!\top}Z\|_{\mathrm{Fr}}^2=\tr(Z^{\!\top}\Pib Z)=\|b^{\!\top}Z\|_{\mathrm{Fr}}^2$.  By Step~3, $\xi_j^{(p)}$ is bounded and $\xi_j^{(p)}\xi_j^{(p)\top}\to c_j^2\,\omega_j\omega_j^{\!\top}$, so
$\langle\xi_j^{(p)},\omega_j^{(p)}\rangle^2=\tr\big(\xi_j^{(p)}\xi_j^{(p)\top}\omega_j^{(p)}\omega_j^{(p)\top}\big)\to c_j^2$, and therefore $\langle h_j,h_j^0\rangle^2\to c_j^2$.  Since $h_j^0\in\mathcal B$, $\langle h_j,h_j^0\rangle=\langle\Pib h_j,h_j^0\rangle$, and by \eqref{eq:floor}
\[
  \cos^2\!\angle_p(\Pib h_j,h_j^0)=\frac{\langle h_j,h_j^0\rangle^2}{\|\Pib h_j\|^2}\longrightarrow1 .
\]
Finally, by \eqref{eq:thmfloor} and \eqref{eq:noisefree}, the bracketed expression in Remark~\ref{rem:noisefree} converges to $(1-c_j^2)+c_j^2\sin^2\!\angle(\omega_j,\beta_j)$, which is also the limit of $\sin^2\!\angle_p(h_j,b_j)$ by Theorem~\ref{thm:main}(i).
\hfill$\qed$

\medskip

\begin{remark}[Coordinate invariance]\label{rem:invariance}
The left side of \eqref{eq:thm} depends only on the data $Y$ and the systematic covariance $\Sigma_0$, whereas the proof works in the loading frame $\bar b$ and the right side is written in terms of $G_B$, $\Sigma_f$ and $F$, which are not separately identifiable (Section~\ref{sec:fac-mod}).  The two descriptions are consistent, as the following direct computation shows.
Fix an invertible $R\in\R^{k\times k}$ not depending on $p$, and replace
\[
  (B,\;F,\;\Sigma_f)\quad\text{by}\quad(BR,\;R^{-1}F,\;R^{-1}\Sigma_fR^{-\top}).
\]
This leaves $Y$, $\Sigma_0$, $\mathcal B$, $\Pib$, $b_j$, $h_j$ and $(\theta_j^{(p)},w_j^{(p)})$ unchanged, and Assumptions~\ref{asm:fm}--\ref{asm:reg} continue to hold, with $G_B$ replaced by $R^{\!\top}G_BR$.  The matrices
\[
  Q_p:=(G_B^{(p)})^{1/2}R\,\big(R^{\!\top}G_B^{(p)}R\big)^{-1/2},
  \qquad
  Q:=G_B^{1/2}R\,\big(R^{\!\top}G_BR\big)^{-1/2}
\]
are orthogonal, and
\[
  \bar b\mapsto\bar b\,Q_p,\quad
  \beta_j^{(p)}\mapsto Q_p^{\!\top}\beta_j^{(p)},\quad
  \xi_j^{(p)}\mapsto Q_p^{\!\top}\xi_j^{(p)},\qquad
  W_0\mapsto W_0,\quad
  M_f\mapsto Q^{\!\top}M_fQ,\quad
  M_F\mapsto Q^{\!\top}M_FQ .
\]
Thus every inner product in the proof is unchanged; $\lambda_j$ is unchanged; and $\omega_j$, $\beta_j$ are replaced by $\pm Q^{\!\top}\omega_j$, $\pm Q^{\!\top}\beta_j$, so $\angle(\omega_j,\beta_j)$ is unchanged.  The parameters $(G_B,\Sigma_f)$ therefore enter Theorem~\ref{thm:main} only through the pair $(M_F,M_f)$ up to simultaneous orthogonal conjugation.
\end{remark}

 \section{Estimable and non-estimable errors}
\label{sec:data-driven}

The identity \eqref{eq:thmsplit}  expresses the error $\sin^2 \angle_p(h_j,b_j)$ in an estimated principal direction as a sum of two terms, each of which has an almost sure limit for a fixed number of observations $n$ as the number
of variables $p$ tends to infinity.

In this section, we show that the two terms differ in character:  the out-of-subspace error  $\operatorname{OE}_j^{(p)}=\sin^2\!\angle_p(h_j,\mathcal B)$, which is a lower bound for the total error,  is {\it estimable} in the sense that it can be estimated from $n$ observations of $p$ variables, and the almost sure limits in $p$ of the estimate and the error agree (Theorem~\ref{thm:obsfloor}).  In contrast, the rotation error $\operatorname{RE}_j^{(p)}=\sin^2\!\angle_p\big(\Pib h_j,\,b_j\big)$, and with it the in-subspace error, is {\it non-estimable} from data:  it cannot be estimated solely from {the observations $Y = Y^{(p)}$} (Theorem~\ref{thm:rotrange} and Corollary~\ref{cor:noest}).
Section~\ref{subsec:usage} summarizes what a user needs to know, beyond the data, to quantify the full error.

\subsection{Out-of-subspace error is estimable from data}

An estimate of out-of-subspace error is expressed in terms of  the $n$ eigenvalues 
$$\theta_1^{(p)} \ge  \theta_2^{(p)}   \ge \cdots\ge\theta_n^{(p)}$$
 of the $n \times n$ scaled dual sample covariance matrix $W^{(p)}=Y^{\!\top}Y/(np)$.

The  { \it average bulk eigenvalue} is the average of the  $n-k$ trailing
 sample eigenvalues:
\begin{equation}\label{eq:elldef}
 \ell^{(p)}:=
   \frac{1}{\,n-k\,}\Big(\tr W^{(p)}-\sum_{i=1}^{k}\theta_i^{(p)}\Big),
\end{equation}
where $\tr W^{(p)}=\|Y\|_{\mathrm{Fr}}^2/(np)$.  Both $\ell^{(p)}$ and $\theta_j^{(p)}$
are computable from the data (given the factor count $k$).  A number of references, such as \cite{yata2012} for iid data in the HL framework, show that $\theta^{(p)}_j - \ell^{(p)}$ is a consistent estimator of the eigenvalue $\lambda_j$ (in the sense that the ratio tends to 1). Here, we are focused on eigenspaces.

In the next theorem, we show that a data-driven quantity, the ratio of the average bulk eigenvalue to the $j$th eigenvalue of $W^{(p)}$, has the same almost sure limit in $p$ as the $j$th out-of-subspace error $\sin^2\!\angle_p(h_j,\mathcal B)$.
\medskip

\begin{theorem}[A data-driven estimate for out-of-subspace error]\label{thm:obsfloor}
Under Assumptions~\ref{asm:fm}--\ref{asm:reg}, conditional on $F$ and
almost surely as $p\to\infty$, for each $j\in\{1,\dots,k\}$,
\begin{equation}\label{eq:obsfloor}
  \limp \;  \frac{\ell^{(p)}}{\theta_j^{(p)}}
   \;=\;
   \frac{\delta^2}{\,n\lambda_j+\delta^2\,}
   \;=\;\limp  \sin^2\!\angle_p(h_j,\mathcal B) = \operatorname{OE}_j.  
\end{equation}

 Hence
\begin{equation}\label{eq:obsfloorineq}
\limp \; \frac{\ell^{(p)}}{\theta_j^{(p)}} \;\le\;\limp\sin^2\!\angle_p(h_j,b_j),
\end{equation}
with equality precisely when the in-subspace rotation 
$\sin^2\!\angle(\omega_j,\beta_j)$ vanishes.

\end{theorem}

\begin{remark}
    This means $\ell^{(p)}/\theta_j^{(p)}$ is a strongly consistent, fully observable
estimator of out-of-subspace error $ \sin^2\!\angle_p(h_j,\mathcal B)$, an asymptotic lower bound for the total error.
\end{remark}

\begin{proof}
By Proposition~\ref{prop:dual}(c), $W^{(p)}\to W$ in spectral norm,
a.s.  The limit $W=F^{\!\top}G_BF/n+(\delta^2/n)I_n$ of
\eqref{eq:Wlim} has top $k$ eigenvalues $\lambda_j+\delta^2/n$ and
smallest eigenvalue $\delta^2/n$ with multiplicity $n-k$.
 Weyl's inequality
gives, for every $i$,
$|\theta_i -\operatorname{eig}_i(W)|\le\|W^{(p)}-W\|\to0$,
where $\operatorname{eig}_i(W)$ denotes the $i$th eigenvalue of $W$.
In
particular $\theta_i\to\delta^2/n$ for each $i\in\{k+1,\dots,n\}$, and
averaging these $n-k$ limits yields $\ell^{(p)} \to\delta^2/n$.  The
convergence $\theta_j\to\lambda_j+\delta^2/n>0$ for
$j=1,\dots,k$ is Proposition~\ref{prop:dual}(d).  As the denominator limit $\theta_j^{(p)}\to\lambda_j+\delta^2/n$ is
strictly positive,
\begin{equation*}
  \frac{\ell^{(p)}}{\theta_j^{(p)}}
   \longrightarrow
   \frac{\delta^2/n}{\lambda_j+\delta^2/n}
   =\frac{\delta^2}{n\lambda_j+\delta^2},
\end{equation*}
which equals $\limp\|\Pibp h_j\|^2$ by \eqref{eq:floor}, proving
\eqref{eq:obsfloor}.  Finally, the inequality in \eqref{eq:obsfloorineq} is immediate from
\eqref{eq:thm}, whose in-subspace summand is nonnegative and vanishes if and only if
$\sin^2\!\angle(\omega_j,\beta_j)=0$.
\end{proof}

Aggregating the per-direction floors over $j$ yields a subspace-level
statement.  Let $H:=[\,h_1\ \cdots\ h_k\,]\in\R^{p\times k}$ collect the top
$k$ sample eigenvectors, $\mathcal{H} := \col(H)$, and let $\Pi_{\mathcal{H}}:=HH^{\!\top}$ and
$\Pi_{\mathcal B}:=
\Pib=\bar b\,\bar b^{\!\top}\,(=bb^{\!\top}$ by Proposition~\ref{prop:chart}(iii)$)$ be the orthogonal projectors onto $\mathcal H$ and
$\mathcal B$, respectively.
\medskip

\begin{corollary}[Aggregate out-of-subspace error]\label{cor:subspace}
Under Assumptions~\ref{asm:fm}--\ref{asm:reg}, conditional on $F$ and
almost surely as $p\to\infty$, the distance between the subspaces $\mathcal{H}$ and $\mathcal{B}$ is equal to the sum of the out-of-subspace floors of Theorem~\ref{thm:main}. That is,
\begin{equation}\label{eq:subspace}
  \sum_{j=1}^{k}\sin^2\!\angle_p(h_j,\mathcal B) \;=\; \tfrac12\,\|\Pi_{\mathcal H}-\Pi_{\mathcal B}\|_{\mathrm{Fr}}^2
   \;\longrightarrow\;
   \sum_{j=1}^{k}\frac{\delta^2}{\,n\lambda_j+\delta^2\,},
\end{equation}
where the first
expression is the sum of the squared sines of the $k$ principal
angles between ${\mathcal H}$ and $\mathcal B$.

By
Theorem~\ref{thm:obsfloor}, the limit is consistently estimated by the
observable statistic $\sum_{j=1}^{k}\ell^{(p)}/\theta_j^{(p)}$.
\end{corollary}

\begin{proof}
Since $\sin^2\!\angle_p(h_j,\mathcal B)=\|\Pibp h_j\|^2$, the limit in
\eqref{eq:subspace} is the sum over $j=1,\dots,k$ of the limits
\eqref{eq:floor} established in Step~2 of the proof of
Theorem~\ref{thm:main}.  For the first equality, expand
\begin{eqnarray}
    \|\Pi_{\mathcal H}-\Pi_{\mathcal B}\|_{\mathrm{Fr}}^2&=&\tr \Pi_{\mathcal H}+\tr \Pi_{\mathcal B}-2\tr(\Pi_{\mathcal H}\Pi_{\mathcal B}) \nonumber \\
&=&2k-2\sum_{j}h_j^{\!\top}\Pib h_j
=2\sum_{j}\big(1-\|\Pib h_j\|^2\big)=2\sum_{j}\|\Pibp h_j\|^2,
\end{eqnarray}
using $\|A\|_{\mathrm{Fr}}^2 = \tr A^2$ for any symmetric matrix $A$, and that $H$ and $b$ each have $k$ orthonormal columns.

 For the
principal-angle identification: the cosines of the principal angles $\eta_i$ between
the $k$-dimensional subspaces $\col(H)$ and $\mathcal B$ are the singular
 values $\sigma_1,\dots,\sigma_k$ of $b^{\!\top}H$ \citep[\S6.4.3]{golubvanloan2013},
 so
$$\sum_{i}\sin^2\eta_i=k-\sum_i\sigma_i^2=k-\|b^{\!\top}H\|_{\mathrm{Fr}}^2
=\sum_{j}\big(1-\|\Pib h_j\|^2\big) = \tfrac12\,\|\Pi_{\mathcal H}-\Pi_{\mathcal B}\|_{\mathrm{Fr}}^2.$$
The final claim
combines \eqref{eq:subspace} with \eqref{eq:obsfloor} summed over $j$.
\end{proof}

\begin{remark}[Inconsistent subspace estimation]
Since each out-of-subspace term is strictly positive, the limit in \eqref{eq:subspace} is
strictly positive, which means the sample subspace $\col(H)$ is not a consistent
estimator of $\mathcal B$ in the Frobenius norm in $p$.
Instead, Corollary \ref{cor:subspace} provides us with a consistent estimator of the distance between the subspaces, as measured by the left hand side of \eqref{eq:subspace}.
\end{remark}

\subsection{Rotation error is non-estimable {from data}}\label{subsec:nonestimable}

{
The exact split of Section \ref{subsec:split} expresses the
estimation error $\sin^2\!\angle_p(h_j,b_j)$ in terms of two angles: the
out-of-subspace angle $\angle_p(h_j,\mathcal B)$ between $h_j$ and the
$k$-dimensional loading subspace $\mathcal B\subset\R^p$, and the in-subspace
angle $\angle_p(\Pib h_j,b_j)$ between the target $b_j\in\mathcal B$ and the
projection $\Pib h_j\in\mathcal B$, which is an angle of rotation inside
$\mathcal B$.  We call $\sin^2\!\angle_p(\Pib h_j,b_j)$ the \emph{rotation
error}.  The previous subsection showed that the out-of-subspace error is
estimable.  Here we show that the rotation error is not, in a strong form: the
data do not merely fail to determine the rotation error; in the absence of further distributional assumptions on $F$, they place no
restriction on it at all.

The data matrix $Y=BF+Z$ is assembled from the loadings $B$, the realized
factor path $F$, and the realized specific returns $Z$; the factor covariance
$\Sigma_f=\E[f^{(l)}f^{(l)\top}]$ is not on the list.   $\Sigma_f$ is
a feature of the \emph{law} of the factor path rather than of the path itself,
and it enters our analysis of $b$ at exactly one point: the estimation targets
$b_1,\dots,b_k$ are the eigenvectors of $\Sigma_0=B\Sigma_fB^{\!\top}$, as in
\eqref{eq:sdt}.  Estimator and target are therefore driven by different
inputs, and, when $n$ is fixed, can be moved independently of one another.

Concretely, fix $B$, $F$ and $Z$, and let $\Sigma$ be an arbitrary positive
definite $k\times k$ matrix.  
Assumption~\ref{asm:fm} requires only that the
marginals of the factor path share the second moment $\Sigma$; taking
$f^{(1)},\dots,f^{(n)}$ i.i.d.\ $\mathcal N(0,\Sigma)$ meets that requirement.  These laws have
positive densities on $\R^{k\times n}$ and are therefore mutually absolutely
continuous, so no event concerning the path $F$ is possible under one such
$\Sigma$ and impossible under another.  Except for Assumption \ref{asm:sep}, the
remaining assumptions do not mention $\Sigma$ at all. We accommodate Assumption~\ref{asm:sep}  by holding the
population spectrum $\mu_1,\dots,\mu_k$ fixed, as follows.

Fixed choices of $B, F$, $Z$, and the eigenvalues $\mu_1,\dots,\mu_k$ of $M_f$ are compatible with a class of models
corresponding to different choices for $\Sigma_f$. To state the next theorem, we need two definitions.
First, for a positive definite
$\Sigma\in\R^{k\times k}$ in the role of the factor covariance $\Sigma_f$, let
$\beta_j(\Sigma)$ be the $j$th unit eigenvector of $M(\Sigma):=G_B^{1/2}\Sigma G_B^{1/2}$, so that $\beta_j=\beta_j(\Sigma_f)$ (see \eqref{eq:Mdef}).
The loadings $B$, the subspace
$\mathcal B=\col(B)$, the limit $G_B$, the realized path $F$, the data
$Y=BF+Z$, the matrix $M_F=G_B^{1/2}(FF^{\!\top}/n)G_B^{1/2}$ and its eigenvectors $\omega_j$ do not vary with $\Sigma$.  (The principal frame $b$ and the scores
$\Phi$ move with $\Sigma$, but their product $b\Phi=BF$ does not.)  In the rotation error $\sin^2\angle(\omega_j,\beta_j)$, therefore, $\Sigma_f$ enters only through the population direction $\beta_j$.

Second, call $\Sigma$ \emph{admissible}
if it is positive definite and
$M(\Sigma)$ has the same eigenvalues
$\mu_1>\cdots>\mu_k>0$ as $M_f = M(\Sigma_f)$.
The admissible $\Sigma$ are the possible covariance matrices for $f$ that do not change our fixed parameters, and are characterized in Lemma~\ref{lem:orbit} of Appendix~\ref{sec:orbit}.

\medskip

\begin{theorem}[Rotation error is not determined by the data]
\label{thm:rotrange}
Let $k\ge2$ and fix $j\in\{1,\dots,k\}$.  Hold fixed the loading sequence $B$,
the constant $\delta^2$, the eigenvalues $\mu_1>\dots>\mu_k$ of $M_f$, the
conditional noise law $Q$ of Assumption~\ref{asm:noise}, a factor path $F$
satisfying Assumption~\ref{asm:reg}, and specific returns $Z$ for which the
conclusions {\rm(a)} and {\rm(b)} of Proposition~\ref{prop:dual} hold, as they
do with $Q(F,\cdot)$-probability one.  Suppose
Assumptions~\ref{asm:delta} and~\ref{asm:gram} hold.  Let
$\Sigma$ range over the admissible factor covariances.  Then
\begin{enumerate}
\item[(i)] the data $Y$, the sample directions $h_j$, the realized eigenvalues
$\lambda_j$, and hence the out-of-subspace error \eqref{eq:thmfloor},
are the same for every admissible $\Sigma$;
\item[(ii)] the rotation error
$\sin^2\!\angle\big(\omega_j,\beta_j(\Sigma)\big)$ attains every value in
$[0,1]$; and consequently
\item[(iii)] the limiting estimation error \eqref{eq:thm} attains every value
in $\big[\,\delta^2/(n\lambda_j+\delta^2),\ 1\,\big]$.
\end{enumerate}
\end{theorem}

The proof is given in Appendix~\ref{sec:orbit}.  Because $Q(F,\cdot)$ meets the requirements of Assumption~\ref{asm:noise} at
every path $F$, that assumption holds, with the same noise law, whatever the
law of the factor path, and in particular under every admissible $\Sigma$.
Moreover the conclusions (a) and (b) of Proposition~\ref{prop:dual} do not
involve $\Sigma$, since $\|b^{\!\top}Z\|_{\mathrm{Fr}}=\|\Pib Z\|_{\mathrm{Fr}}$
depends on the frame $b$ only through $\col(B)$.  The limits of
Theorem~\ref{thm:main} therefore hold for the fixed pair $(F,Z)$
simultaneously for every admissible $\Sigma$. For $k=1$ the conclusion fails: the only orthogonal $1\times1$ matrices are $\pm1$, and the rotation
error vanishes identically (Lemma~\ref{lem:orbit}).

Theorem~\ref{thm:rotrange} is a statement about a single realization: the data
place no restriction whatever on the limiting rotation angle, even if the
factor path $F$ is added to the available data.  Since the factor laws
compared are mutually absolutely continuous, it follows that no consistent
estimator of the rotation error exists.

\medskip
\begin{corollary}[Rotation error cannot be estimated from data alone]
\label{cor:noest}
Let $k\ge2$ and fix $j\in\{1,\dots,k\}$.  Hold fixed $B$, $\delta^2$, the eigenvalues
$\mu_1>\dots>\mu_k$ and the conditional noise law $Q$, with
Assumptions~\ref{asm:noise}--\ref{asm:gram} in force.  For each admissible
$\Sigma$ let $\PP_\Sigma$ be the joint law of $(F,Z)$ when
$f^{(1)},\dots,f^{(n)}$ are i.i.d.\ $\mathcal N(0,\Sigma)$ and $Z$ given $F$
has law $Q(F,\cdot)$.  Then there is no sequence of measurable functions
$T_p=T_p(Y^{(p)},F)$, not depending on $\Sigma$, such that
\[
  T_p-\sin^2\!\angle_p\big(\Pib h_j,\,b_j\big)\longrightarrow0
  \quad\text{in $\PP_\Sigma$-probability, for every admissible $\Sigma$.}
\]
In particular no such estimator exists that is a function of the data
$Y^{(p)}$ alone, and none exists over any larger class of factor laws
permitted by Assumption~\ref{asm:fm}.
\end{corollary}

The proof is given in Appendix~\ref{sec:orbit}.
Additional information, such as a probability model for the path $F$, is needed to draw any conclusions about the rotation error (Remark~\ref{rem:identification}).

The deficiency is one of sample size $n$ rather than cross-sectional size $p$.
Growing the cross-section supplies no further draws of $f$ and so says nothing
about the population $\Sigma_f$ against which that realization is being
compared.  By Proposition~\ref{prop:n-limit}, the rotation error vanishes in the
large-$n$ limit.

}

\begin{remark}[Non-identification, not uninformativeness]\label{rem:identification}
Theorem~\ref{thm:rotrange} is an identification statement about a fixed data set: every value of the rotation error in $[0,1]$ is consistent with the model.  It does not preclude probabilistic statements about the rotation error once a probability model for the factor path is adopted.  Under such a model the rotation error $\operatorname{RE}_j=\sin^2\!\angle(\omega_j,\beta_j)$ is a random variable whose law is determined by the law of $M_F$ relative to the fixed matrix $M_f$, and, in the limit, the data reveal the eigenvalues $\lambda_j$ of $M_F$ (by the proof of Theorem~\ref{thm:obsfloor}, $\theta_j^{(p)}-\ell^{(p)}\to\lambda_j$).  
\end{remark}

\subsection{What is needed to quantify the full error}\label{subsec:usage}

Theorems~\ref{thm:main}--\ref{thm:rotrange} determine what a user must know, beyond the data, to quantify the principal direction error $\angle_p(h_j,b_j)$.
\begin{enumerate}
\item \emph{Data alone.}  The out-of-subspace error, and hence a lower bound for the total asymptotic error, is consistently estimated by $\ell^{(p)}/\theta_j^{(p)}$ (Theorem~\ref{thm:obsfloor}); the aggregate subspace error is estimated by the sum of these ratios (Corollary~\ref{cor:subspace}).  The ingredients $\delta^2$ and $\lambda_j$ are themselves consistently estimated by $n\ell^{(p)}$ and $\theta_j^{(p)}-\ell^{(p)}$.
\item \emph{Data and a model for the law of the factor path.}  The limiting rotation error $\sin^2\!\angle(\omega_j,\beta_j)$ depends on the model parameters only through the pair $(M_F,M_f)$, and is unchanged when both are conjugated by the same orthogonal matrix (Remark~\ref{rem:invariance}).  A model that specifies the law of $M_F$ given $M_f$ therefore determines a law for the rotation error and, with the estimates in item~1, for the total error \eqref{eq:thm}; the loading Gram matrix $G_B$ and the factor covariance $\Sigma_f$ are not needed separately.  
\item \emph{Data and the pair $(M_F,M_f)$.}  If $M_F$ and $M_f$ are known up to simultaneous orthogonal conjugation, then Theorem~\ref{thm:main} together with item~1 gives a consistent estimator of the full error $\sin^2\!\angle_p(h_j,b_j)$.
\end{enumerate}

\section{Simulation} \label{sec:simulation}
In this section we present some simulation experiments intended to illustrate and numerically confirm our theoretical results in a simple setting. We take specific returns to be iid $t$-distributed, and the loading Gram limit $G_B$ to be diagonal.
Our specific aim is to examine empirically how well our theoretical asymptotic limits serve to approximate finite $p$ behavior, a topic of interest that is not treated by the preceding theory. We use volatility parameters guided by representative behavior of US equities.

Section~\ref{subsec:calibration} describes the experimental setup and parameters; \ref{subsec:validation} examines the convergence described in Theorem~\ref{thm:main}; \ref{subsec:validation-2} examines the convergence of Theorem~\ref{thm:obsfloor}; \ref{subsec:varyn} looks at outcomes across various typical values of $n$, the number of observations.

\subsection{Model calibration and experimental design} \label{subsec:calibration}
Guided by the empirical analysis in \cite{bayraktar2014}, we specify
a
three-factor instance of the model in~\eqref{eq:lfm}  based on the US
public equity market. We assume factor and specific returns are
independently and identically $t$-distributed over time with mean $0$, and $6$ and $5$
degrees of freedom, respectively.  Their annualized volatilities are in
Table~\ref{tab:volatilities}.

\begin{table}[h]
\centering
\caption{Annualized factor and specific volatilities}
\label{tab:volatilities}
\vspace{.1cm}
\begin{tabular}{lcccc}
\toprule
& Factor 1 & Factor 2 & Factor 3 & Specific \\
\midrule
Volatility & 0.16 & 0.08 & 0.06 & 0.40 \\
\bottomrule
\end{tabular}
\end{table}

We run simulations, each with $1000$
 paths.  In some experiments, we keep the number of observations $n$ fixed and grow the number of variables $p$.  In others, the reverse.

The simulation is implemented as follows. We choose $\nmax = 250$ days and $\pmax = 50000$ assets. We first form a $\pmax \times 3$ matrix $B$
using independent draws from the
 normal distributions $\mathcal N(1, 0.25)$, $\mathcal N(0,1)$, and $\mathcal N(0,1)$ for the three columns. This loading matrix $B$ is fixed for
 all 1000 paths, and corresponds to a loading Gram limit $G_B$:
 \begin{align*}
    G_B = \begin{pmatrix}
        1.25 & 0 & 0\\
        0& 1 & 0\\
        0 & 0 & 1
    \end{pmatrix}.
\end{align*}

For each of the 1000 simulation paths
we form
a $3 \times \nmax$ factor return matrix $F$ and a $\pmax \times \nmax$ specific return matrix $Z$.
The rows of $F$, corresponding to factors 1,2, and 3, are drawn iid from mean zero $t$-distributions as described above.  All entries of $Z$ are drawn iid from a mean zero $t$-distribution with 5 degrees of freedom and annualized standard deviation 0.4.

This provides us, for each simulation path, a $\pmax \times \nmax$ data matrix $Y$ determined by
\begin{equation}
    Y = BF + Z.
\end{equation}
For each path we may vary $p \leq \pmax$ and $n \leq \nmax$ by forming the nested $p \times n$ matrices $Y = Y^\pn$ from the first $p$ rows and first $n$ columns of $Y$.  In this way, for each of the 1000 simulations we have a family of data matrices $Y^\pn$ indexed by $(p,n)$, and hence sample covariance matrices with their eigenvalues and eigenvectors, for analysis in the next sections.

The experiments that follow depend on a single draw of factor exposures $B$.  While this choice does not affect asymptotic results so long as the assumptions in Section~\ref{sec:data} are satisfied, it has an impact on results for small $p$.

\subsection{Asymptotics of principal direction error guaranteed in Theorem~\ref{thm:main}}\label{subsec:validation}

We next fix the number of observations at $n=63$, corresponding to three months of daily data, and increase $p$ from $100$ to $50,000$ to illustrate the rate at which the errors in principal directions $\sin^2\angle_p(h_j,b_j)$ and their constituents converge to the asymptotic limits of Theorem~\ref{thm:main}.  Simulation results using 1000 independent nested $p$-trajectories are presented in Figure~\ref{fig:decomp_p_sweep_combined}.

The dots in the top panel show the average $\sin^2\angle_p(h_j,b_j)$  for each factor $j=1,2,3$, which corresponds to the left hand side of equation~\eqref{eq:thm}.   The blue fill represents the asymptotic out-of-subspace error (the $\operatorname{OE}_j$ of the Introduction; we suppress the index $j$)
$$ \operatorname{OE} =
\frac{\delta^2}{n\lambda_j+\delta^2}
$$
and the orange represents the asymptotic in-subspace error
$$ \operatorname{IE} =
\frac{n\lambda_j}{n\lambda_j+\delta^2}\,
              \sin^2\!\angle(\omega_j,\,\beta_j).
$$
These component errors sum to total error:
$$ \operatorname{TE} =  \operatorname{OE} +\operatorname{IE}.$$

Combined, we see that the dots converge to the sum
$ \operatorname{TE} = \operatorname{OE}+\operatorname{IE}$ of the blue and orange sections.  At $p=500$, the errors are somewhat large, translating $\angle_p(h_j,b_j)$ to $17^\circ$ for Factor 1, $39^\circ$ for Factor 2 and $50^\circ$ for Factor 3.

Also note that the out-of-subspace error (blue) comprises the majority  of the total asymptotic error (blue plus orange) for each factor. This indicates that our estimation error is primarily due to the influence of the idiosyncratic returns on the principal components, pushing them away from the span of the population systematic factors.

The bottom panel of Figure~\ref{fig:decomp_p_sweep_combined} shows the distribution of the pathwise differences of each of these elements, subtracting finite-$p$ values from their asymptotic limits:
\begin{itemize}
\item grey boxes: $\operatorname{TE} - \sin^2 \angle_p (h_j, b_j) $
\item blue boxes: $\operatorname{OE} -\sin^2\!\angle_p(h_j,\mathcal B) $
\item orange boxes: $\operatorname{IE} - \cos^2\!\angle_p(h_j,\mathcal B) \sin^2\!\angle_p\big(\Pib h_j,\,b_j\big) $
\end{itemize}

This figure suggests that the asymptotic value is an underestimate of the true finite-$p$ error. This holds for the choice of factor exposures $B$ in the experiment reported here; it need not hold in general.
Analogous comments apply to all figures in this section.

\begin{figure}[htbp]
    \centering
    \includegraphics[width=\textwidth]{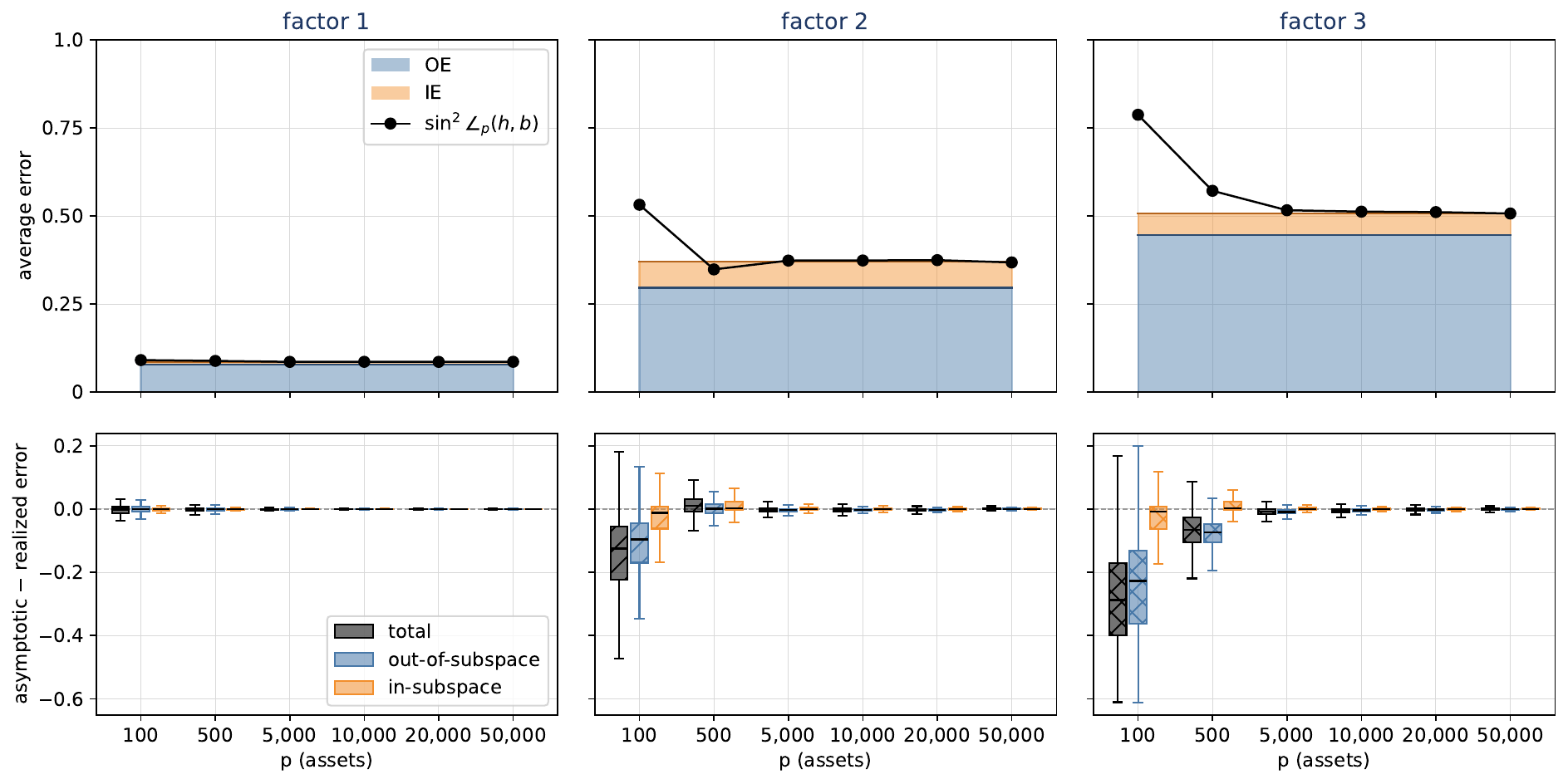}
   \caption{Error in principal directions estimated from $n=63$ observations based on 1000 simulations. The $x$-axis is the number of variables $p$.  Top panel:  The $y$-axis is average error in units of $\sin^2$ across the $1000$ simulations. Dots are average total error $\sin^2 \angle_p (h_j, b_j)$ for each $p$. Blue and orange fill represent average asymptotic limits of out-of-subspace ($\operatorname{OE}$)  and in-subspace error ($\operatorname{IE}$), which sum to asymptotic limits of total error ($\operatorname{TE}$). Bottom panel:  The $y$-axis is path-by-path differences between the asymptotic limit and the finite-$p$ error, with grey bars for $\operatorname{TE} -\sin^2\!\angle_p(h_j, b_j)$, blue bars for  $\operatorname{OE} - \sin^2\!\angle_p(h_j,\mathcal B)$, and orange bars for $\operatorname{IE} - \cos^2\!\angle_p(h_j,\mathcal B) \sin^2\!\angle_p\big(\Pib h_j,\,b_j\big)$. Differences must tend to zero by Theorem~\ref{thm:main}.}
   \label{fig:decomp_p_sweep_combined}
\end{figure}

Figure~\ref{fig:rot_p_sweep} summarizes simulation results for the rotation error $\sin^2\!\angle_p\big(\Pib h_j,\,b_j\big)$, which
reflects sampling error in factor returns.  This error persists in the $p$-asymptotic limit when $n$ stays fixed.

The dots in the top panel show the average   $\sin^2\!\angle_p\big(\Pib h_j,\,b_j\big)$
for each factor $j=1,2,3$, which corresponds to the left hand side of \eqref{eq:thmrot}.  The green fill represents the asymptotic rotation error
$$ \operatorname{RE} =
  \sin^2\!\angle\big(\omega_j,\,\beta_j\big),
$$
which is given by the right hand side of \eqref{eq:thmrot}.

The bottom panel of Figure~\ref{fig:rot_p_sweep} shows the distribution of pathwise differences, subtracting finite-$p$ values from their asymptotic limits:

\begin{itemize}
\item green boxes: $ \operatorname{RE}  - \sin^2\!\angle_p\big(\Pib h_j,\,b_j\big)$
\end{itemize}

\begin{figure}[htbp]
    \centering
    \includegraphics[width=\textwidth]{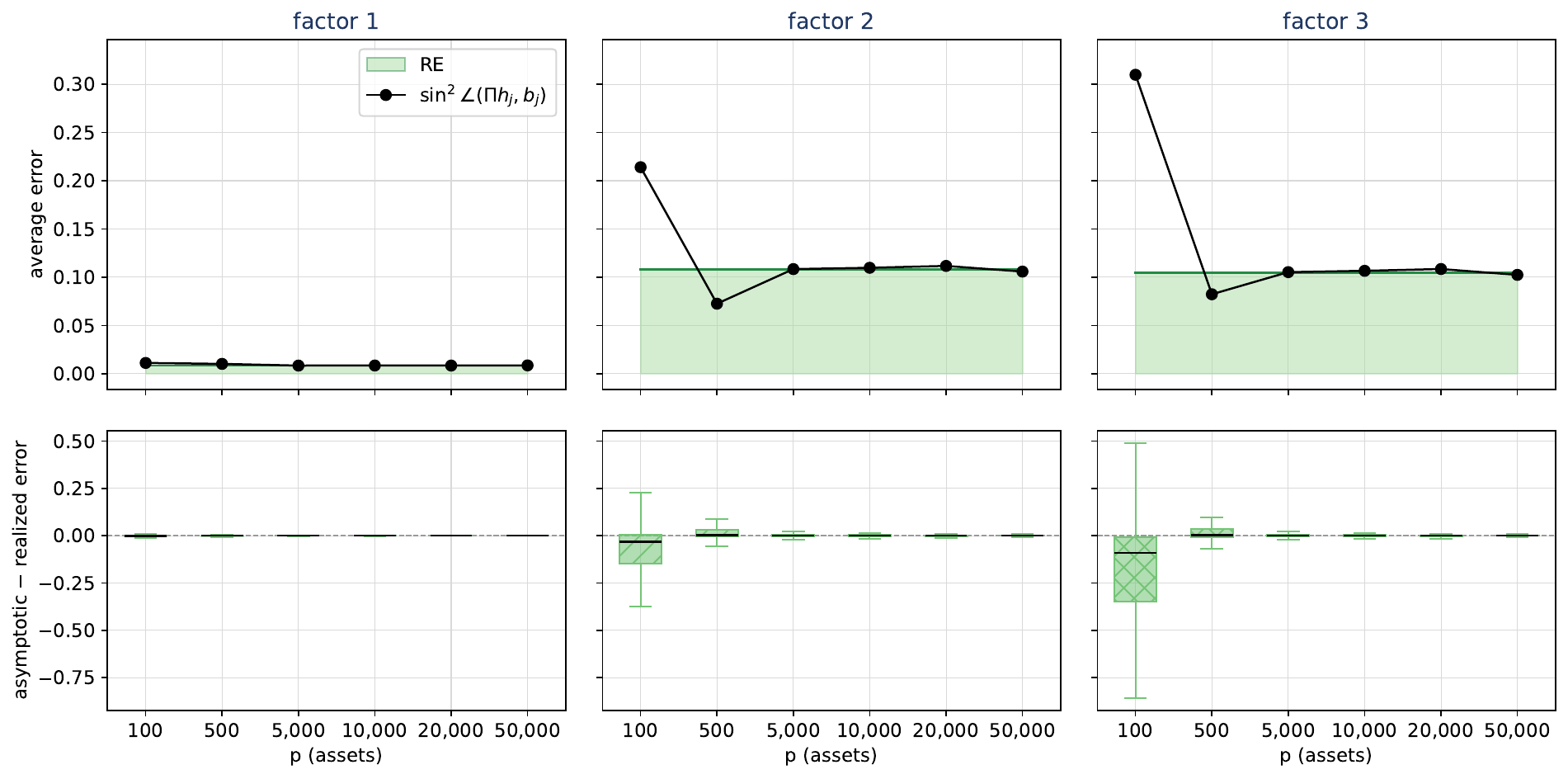}
   \caption{Rotation error in principal directions estimated from $n=63$ observations based on 1000 simulations. The $x$-axis is the number of variables $p$.   Top panel:  The $y$-axis is average error in units of $\sin^2$ across the $1000$ simulations. Dots are average rotation error  $\sin^2\!\angle_p\big(\Pib h_j,\,b_j\big)$ for each $p$.
   Green fill represents average rotation error  asymptotic limit ($\operatorname{RE}$). Bottom panel: The $y$-axis is path-by-path differences between the asymptotic limit  and the finite-$p$ error with green bars for $\operatorname{RE} - \sin^2\!\angle_p\big(\Pib h_j,\,b_j\big)$. Differences must tend to zero by Theorem~\ref{thm:main}.\label{fig:rot_p_sweep}}
\end{figure}

\subsection{Asymptotics of finite-$p$ estimates of out-of-subspace error guaranteed in Theorem~\ref{thm:obsfloor}}
\label{subsec:validation-2}

Theorem~\ref{thm:obsfloor} states that out-of-subspace error can be estimated in terms of eigenvalues of a sample covariance matrix, and the estimate has the same almost sure limit in $p$ as the error.  For convenience, we repeat the result here, with $\theta_j^{(p)}$ equal to the $j$th eigenvalue of the scaled dual sample covariance matrix $W^{(p)}$ and $\ell = \ell^{(p)}$ its average bulk eigenvalue:
\begin{equation*}
  \limp  \frac{\ell^{(p)}}{\theta_j^{(p)}}
   \;=\;
   \frac{\delta^2}{\,n\lambda_j+\delta^2\,}
   \;=\;\limp  \sin^2\!\angle_p(h_j,\mathcal B).
\end{equation*}
Here we examine the accuracy of the estimate.

As in Subsection~\ref{subsec:validation}, we fix the number of observations at $n=63$ and increase $p$ from $100$ to $50,000$.

The dots in the top panel of Figure~\ref{fig:growing_p_obs_est_floor} show the average out-of-subspace error $ \sin^2\!\angle_p(h_j,\mathcal B)$  for each factor $j=1,2,3$. The stars show its average finite-$p$ estimate 
$ \ell^{(p)}/\theta_j^{(p)}
$   
of out-of-subspace error.
The blue fill represents the average asymptotic out-of-subspace error $\operatorname{OE}$.
We see that both the dots and the stars converge to
$\operatorname{OE}$.  For small $p$, however, the average finite-$p$ estimate and average true finite-$p$ error differ, especially for factors 2 and 3.

The bottom panel of Figure~\ref{fig:growing_p_obs_est_floor} examines this difference in detail.  It shows the distribution of the pathwise differences of the finite-$p$ estimate and error:
\begin{itemize}
\item red boxes: $\ell^{(p)}/\theta_j^{(p)} -\sin^2\!\angle_p(h_j,\mathcal B)$
\end{itemize}

In our experiment, the out-of-subspace errors $\sin^2\!\angle_p(h_j, \mathcal{B})$ were substantial.  Translating from units of sine-squared to angles, we found finite-$p$ averages of $16.3^\circ$, $33.1^\circ$ and $47.0^\circ$ for Factors 1, 2 and 3 at $p = 500$.  This is an essential point, since out-of-subspace error serves as a floor for the total error.   The finite-$p$ estimates $\ell^{(p)}/\theta_j^{(p)}$ tended to  underestimate the true error in our experiment.  However, we have not explored the extent to which this effect is generic.

\begin{figure}[htbp]
    \centering
    \includegraphics[width=\textwidth]{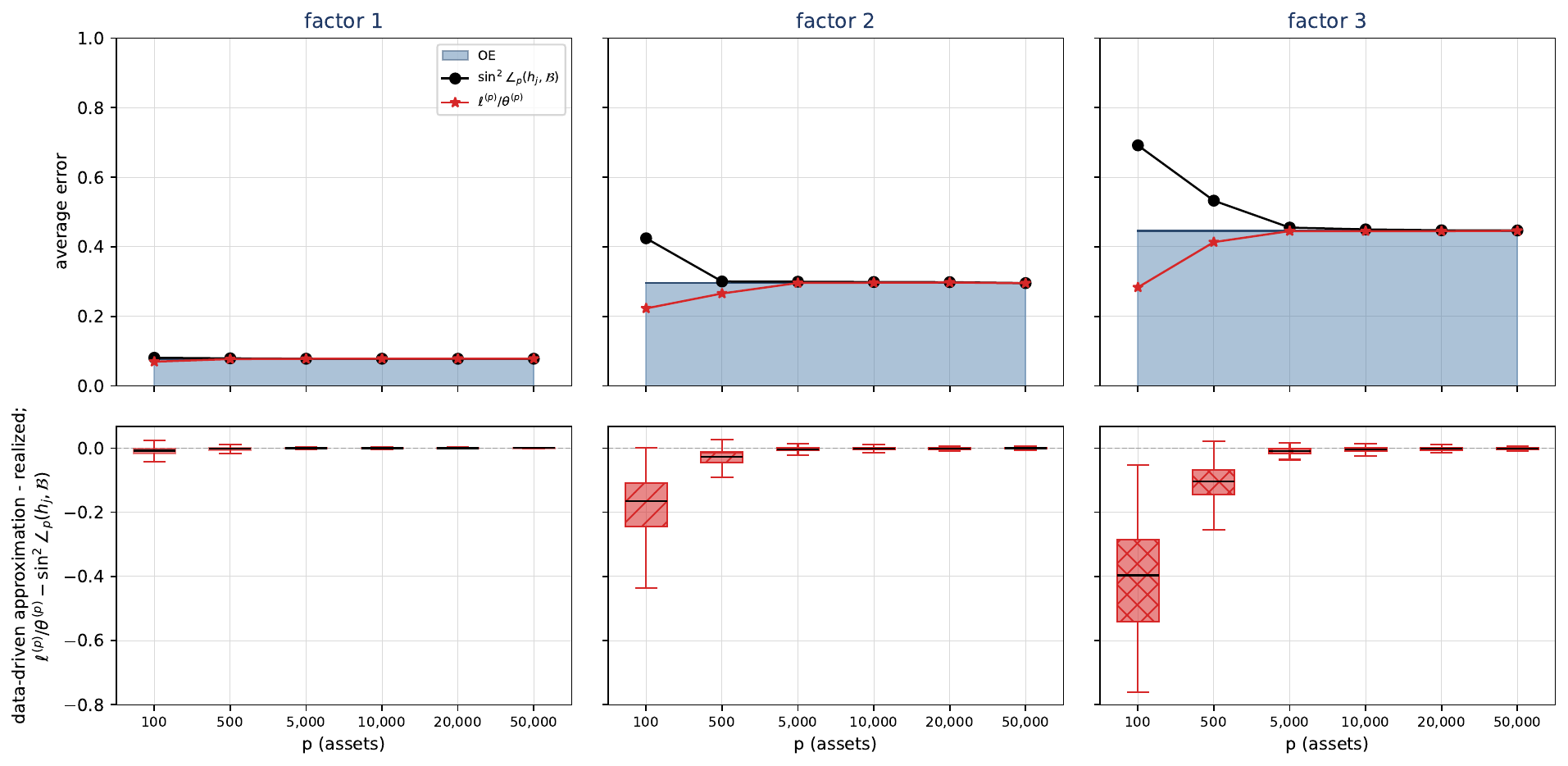}
    \caption{Out-of-subspace error in principal directions estimated from $n=63$ observations based on 1000 simulations.  The $x$-axis is the number of variables $p$. Top panel: $y$-axis is average out-of-subspace error in units of $\sin^2$.  Black dots are average of $\sin^2\angle_p(h_j,\mathcal{B})$ for each $p$.  Red stars are average of data-driven estimates $\ell^{(p)}/\theta_j^{(p)}$ of OE.  Blue fill represents average OE.
    Bottom panel: boxplot of path-by-path differences between data-driven estimates and realized errors, given by $\ell^{(p)}/\theta_j^{(p)} - \sin^2 \angle_p (h_j, \mathcal{B})$ for each $p$.  Differences must tend to zero by Theorem~\ref{thm:obsfloor}.}   
    \label{fig:growing_p_obs_est_floor}
\end{figure}

\subsection{Dependence of the error in principal directions on the number of observations}\label{subsec:varyn}
How much data are required to drive estimation error in principal directions to a reasonable level?
We address the question experimentally by increasing the number of
observations $n$ from $20$ to $250$, with the number of variables $p$ fixed
at $3000$, the approximate number of securities in the Russell 3000
Index. Figure~\ref{fig:decomp_n_sweep} mirrors Figure~\ref{fig:decomp_p_sweep_combined}, with the roles of $p$ and $n$ interchanged.

The top panel of
Figure~\ref{fig:decomp_n_sweep} shows averages over the 1000 simulation paths, with
dots representing finite-$p$ average errors $\sin^2 \angle_p(h_j, b_j)$, and blue and orange bars
representing out-of-subspace and in-subspace components of the asymptotic
error (from the right-hand-side of~\eqref{eq:thm}). Translating from units of
sine-squared to angles, average errors decreased from $30^\circ$ to
$8.5^\circ$ for factor 1, from $56^\circ$ to $19^\circ$ for factor 2,
and from $66^\circ$ to $24^\circ$ for factor 3.  The percentage of
total error (in sine-squared units) accounted for by the average out-of-subspace component of the
asymptotic error is observed to stay
between $80\%$ and $93\%$ for all three factors, across our range of $n$. (We expect these percentages to depend on the specific variance parameters we have chosen; we don't pursue this question.)



The ranges of path-by-path differences between the finite-$p$ errors and
the asymptotic limits shown in the bottom panel of
Figure~\ref{fig:decomp_n_sweep} decrease steadily for all three factors as
$n$ increases, but remain non-negligible even at $n=250$. Asymptotic limits tended to be less than finite-$p$ errors for factors 2 and 3. In our experiment and for this range of $n$ and $p$, the $p$-asymptotic limit of Theorem~\ref{thm:main} is a considerably better approximation of the finite $n,p$ true
error $\sin^2 \angle_p(h_j, b_j)$ than is zero,  the $n$-asymptotic limit. This may be viewed as support for use of the HL asymptotic regime for these
values of $n$ and $p$.

\begin{figure}[htbp]
    \centering
    \includegraphics[width=\textwidth]{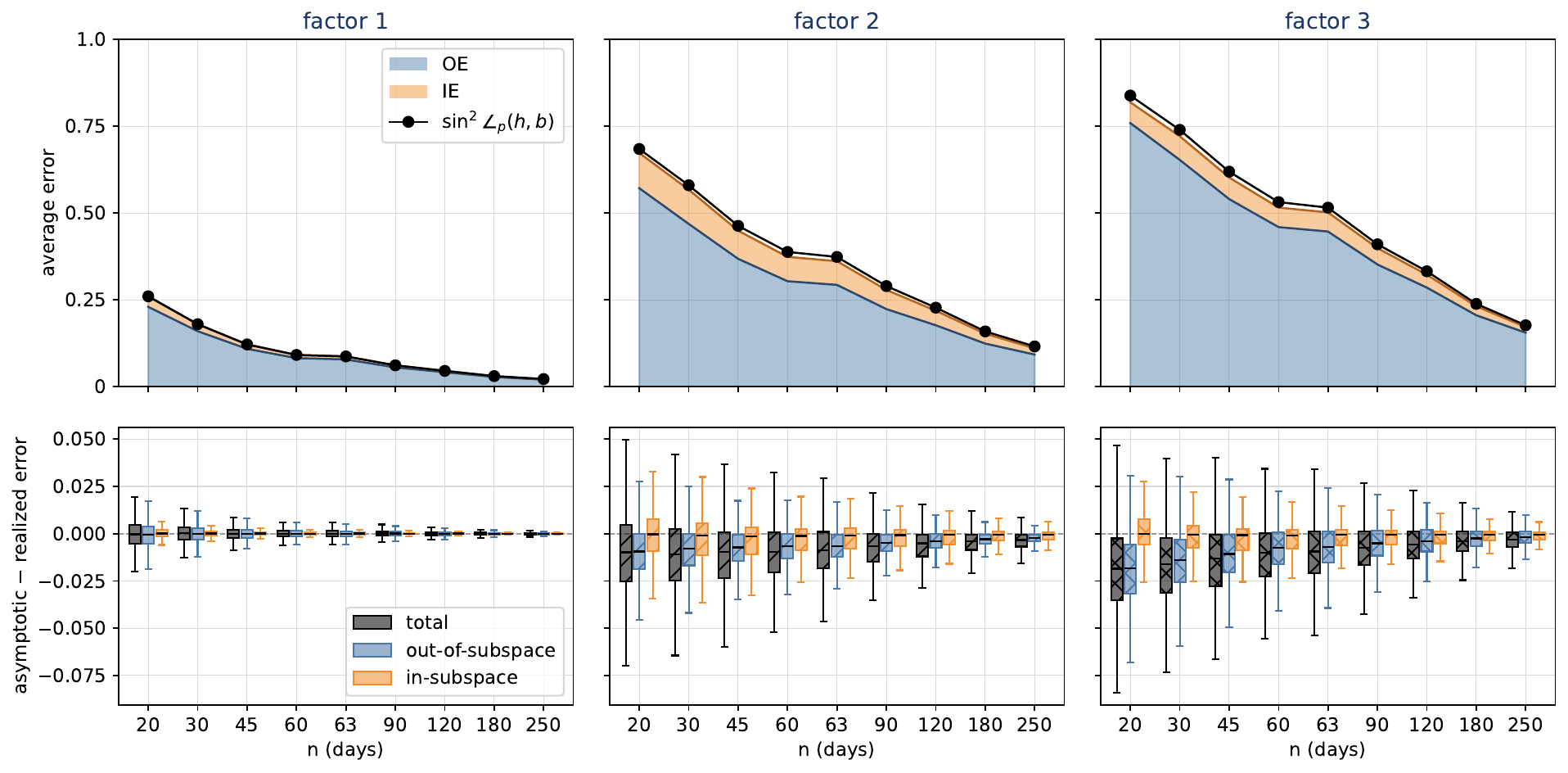}
    \caption{
    \label{fig:decomp_n_sweep}
   Errors in principal directions of $p = 3000$  variables based on 1000 simulations.  The $x$-axis is the number of observations $n$.   Top panel:  The $y$-axis is average error in units of $\sin^2$ across the $1000$ simulations. Dots are average total error $\sin^2 \angle_p (h_j, b_j)$ for each $n$. Blue and orange fill represent average asymptotic limits of out-of-subspace ($\operatorname{OE}$)  and in-subspace error ($\operatorname{IE}$), which sum to asymptotic limits of total error ($\operatorname{TE}$).  Bottom panel: 
   The $y$-axis is path-by-path differences between the asymptotic limit and the finite-$p$ error, with grey bars for $\operatorname{TE} -\sin^2\!\angle_p(h_j, b_j)$.
   }
\end{figure}

\section{Relation to prior work}\label{sec:prior}

With our results in hand, we compare them in detail with the closest antecedents in the HL literature and its large-$n$ counterpart.

\cite{jung2012} provide a joint limiting distribution of the inner products between sample and population eigenvectors in terms of lower dimensional dual Gram matrices similar in spirit to ours.  For the setting of spiked models, their Theorem~2(ii)  contains the structure of our decomposition: in our notation their setting is $G_B=I$ with $\Sigma_f$ diagonal, and the square of their limit factors is the product of a subspace term, $n\lambda_j/(n\lambda_j+\delta^2)$, and the squared cosine of the angle between the $j$th eigenvector of the realized score Gram matrix and the $j$th coordinate axis.  By Remark~\ref{rem:invariance}, applied with $R=G_B^{-1/2}P$, where $P$ is an orthogonal matrix whose columns are the eigenvectors $\beta_j$ of $M_f$, every model satisfying our assumptions is equivalent to one with $G_B=I_k$ and $\Sigma_f$ diagonal, and in that model our rotation term $\sin^2\angle(\omega_j,\beta_j)$ is the squared sine of the angle between the $j$th eigenvector of the realized factor second-moment matrix $FF^{\!\top}/n$ of the new model and the $j$th coordinate axis.  The generality of $G_B$ and $\Sigma_f$ is therefore not where our results depart from theirs.  The differences lie in the assumptions, in the mode of convergence, and in the separation of the two terms by their estimability.
We assume specific returns are independent and identically distributed
across observations, conditional on the factor path $F$, with
cross-sectional dependence limited to an approximate factor structure
(uniformly bounded specific covariance eigenvalues) together with
fourth-moment conditions (Assumption~\ref{asm:noise}).  We leave the
joint law of $F$ unrestricted beyond its second moments, allowing factors
to be serially dependent and non-stationary, as appropriate to financial
data.  Our limits are almost sure and conditional on the realized factor path, whereas the limits of \cite{jung2012} are in distribution.  Neither \cite{jung2012} nor our results require the model variables
to be Gaussian, sub-Gaussian, sub-exponential, or to belong to any
restricted class other than having bounded fourth moments.

A parallel development in \cite{jung2022} applies the
 machinery of \cite{jung2012} to the bias of sample
and prediction principal-component {scores}, which are scaled right eigenvectors of a data matrix.  When that matrix is governed by a spiked model, the scores can be decomposed into a
rotation and an estimable scaling factor, with correction for the latter. Its
eigenvector-angle byproduct matches the structure of our
Theorem~\ref{thm:main}, but is obtained in probability
under i.i.d.\ sampling from a single covariance, whereas we allow for serial correlation and obtain a
path-conditional almost-sure limit.

\cite{shen2016} study a spiked covariance model in the HL regime. Their Theorem~5 establishes almost sure limits for the angles between sample eigenvectors and the subspace spanned by the spiked population eigenvectors, strengthening the convergence in distribution of \cite{jung2012} at the price of stronger distributional assumptions.   They observe that with $n$ fixed the sample eigenvectors within the spike group ``can't be asymptotically distinguished,'' and for this reason measure angles to the subspace only.  Theorem~\ref{thm:rotrange} makes that observation exact.

Theorem 3.5(c)
of \cite{gurdogan2024} anticipates, in a different formulation with different background assumptions, our Theorem~\ref{thm:obsfloor}. 
Our Corollary~\ref{cor:noest} is a noise-inclusive counterpart of their Theorem~4.5, which shows in a noise-free model that there is no measurable consistent estimator of the matrix of inner products between sample and population eigenvectors.  See also \cite{gurdogan_shkolnik2026}.

\cite{wang2017} work in a regime where $n, p \to \infty$ with $p/(n\lambda)$ bounded for each spiked eigenvalue $\lambda$, with i.i.d.\ sub-Gaussian observations. Their decomposition of a sample eigenvector into in-subspace and out-of-subspace components is the split underlying Theorem~\ref{thm:main}. Their Theorem~3.2 gives large-$n$ counterparts of our two error terms: the out-of-subspace component converges in probability to a deterministic limit, while the in-subspace component is asymptotically Gaussian at scale $n^{-1/2}$. Their remark that the two components ``intertwine in such a way that correction for the biases of estimating eigenvectors is almost impossible'' anticipates, informally, our Theorem~\ref{thm:rotrange}. In the HL regime with $n$ fixed, neither term is deterministic and only the first is estimable.

\section{Conclusion}

In this article, we analyzed errors in principal directions, unit-length linear combinations of the exposures of a factor model that explain maximum variance.

Orthogonal decomposition yields an expression of  the error in the $j$th estimated principal direction as a sum of two terms.  The first is out-of-subspace error, the squared sine of the angle between the estimate $h_j$ and the space $\mathcal B$ spanned by the true principal directions. The second is in-subspace error, which is driven by the rotation error, the angle between the projection of the estimate onto $\mathcal B$ and the true principal direction, $b_j$.  We show that each of the terms, and therefore the error itself, has an almost sure limit in the HL regime as the number of variables  tends to infinity.

The  summands differ in character.  The out-of-subspace error can be approximated in terms of the eigenvalues of a sample covariance matrix, and the approximation has the same almost sure limit as the error itself. It serves as a  data-driven lower bound for the total error. In contrast, the rotation error, and with it the in-subspace error,
cannot be estimated from data alone.

In a simulation specified with parameters guided by an empirical analysis of the  US equity market, we found support for our theoretical results and observed:
\begin{itemize}
   \item With even a year's worth of daily data (250 observations), errors in principal directions were large enough to demand attention from researchers who rely on principal component analysis.
    \item The out-of-subspace error dominated the total error.
    \item The data-driven estimate of out-of-subspace error was, at once, large enough to demand attention and typically smaller than the error itself.
\end{itemize}

The quantification of  error in principal directions highlights the utility of HL regime asymptotics.  It can shed light on the vast range of applications of principal component analysis in the physical and social sciences. A researcher using estimated principal directions for financial risk management and portfolio construction, as control vectors in regressions, or in other applications, can use the results in this article to gauge
the magnitude of the finite-$n$ error, which can have profound implications for  economic decisions and scientific conclusions.

\backmatter

\bmhead{Acknowledgements}
We thank Navid Bahadoran, Hubeyb Gurdogan, and Sungkyu Jung  for helpful comments, and Kristen Ho for research support.

\section*{Statements and Declarations}

\bmhead{Funding}
The authors did not receive support from any organization for the submitted work.

\bmhead{Competing interests}
The authors have no relevant financial or non-financial interests to disclose.

\bmhead{Data availability}
No new empirical data were analysed in this study. All results in Section~\ref{sec:simulation} are based on simulated data generated from the model and parameters specified there.

\bmhead{Code availability}
The code used to run our simulations and produce the figures of Section~\ref{sec:simulation} is available at \url{https://github.com/alexbstl/factor_lab_pc_concentration_code}.

\bmhead{Use of AI tools}
This article was written by humans for humans. The authors acknowledge the assistance of the AI tools Claude, ChatGPT, and DeepSeek during the research and writing process.  These tools were used to explore ideas, as a source of suggestions (but not decisions) for  variations and extensions of arguments.  They sometimes assisted with initial draft exposition, as well with literature search, proofreading, notation, latex formatting, and other clerical tasks. The authors are solely responsible for all writing and results in this article.


\begin{appendices}

\section{Supporting results}  \label{sec:lemmas}

The results below provide the technical support for the proof of Theorem~\ref{thm:main} and surrounding discussion.

\subsection{Matrix preliminaries}

\begin{lemma}[Gram duality]\label{lem:gramdual}
Let $A$ be a real matrix.
\begin{enumerate}
\item[(i)] $AA^{\!\top}$ and $A^{\!\top}A$ have the same positive eigenvalues,
with the same multiplicities. Moreover, for $\lambda>0$ and each $j$, $\lambda$
is the $j$th largest eigenvalue (counted with multiplicity) of $AA^{\!\top}$ if
and only if it is the $j$th largest eigenvalue of $A^{\!\top}A$.
\item[(ii)] Suppose $\lambda>0$ is the $j$th largest eigenvalue of one of
$AA^{\!\top}$, $A^{\!\top}A$ and is simple there. Then the same holds for the
other. If $u$ is a unit eigenvector of $A^{\!\top}A$ for $\lambda$, then
$Au/\sqrt\lambda$ is a unit eigenvector of $AA^{\!\top}$ for $\lambda$;
conversely, if $v$ is a unit eigenvector of $AA^{\!\top}$ for $\lambda$, then
$A^{\!\top}v/\sqrt\lambda$ is a unit eigenvector of $A^{\!\top}A$ for
$\lambda$.
\end{enumerate}
\end{lemma}

Part~(i) is standard (e.g.\ \citealp[Thm.~1.3.22]{hornjohnson2013}); part~(ii)
is an eigenvector form of the singular value decomposition
(cf.\ \citealp[\S2.4]{golubvanloan2013}).
The proof is straightforward.

\medskip

\begin{lemma}[Eigenpair convergence]\label{lem:econv}
Let symmetric $A^{(p)}\in\R^{m\times m}$ converge entrywise to $A$, and let
$\zeta$ be a simple eigenvalue of $A$ with unit eigenvector $v$.
\begin{enumerate}
\item[(i)] For large $p$, $A^{(p)}$ has a simple eigenvalue
$\zeta^{(p)}\to\zeta$, whose unit eigenvectors $v^{(p)}$ (regardless of sign choice) satisfy
$\angle(v^{(p)},v)\to0$; equivalently, $|\langle v^{(p)},v\rangle|\to1$.
\item[(ii)] For unit vectors $u,v\in\R^m$,
$\|uu^{\!\top}-vv^{\!\top}\|_{\mathrm{Fr}}^2=2\sin^2\!\angle(u,v)$. Consequently,
for unit vectors $u^{(p)},u$, $\angle(u^{(p)},u)\to0$ if and only if
$u^{(p)}u^{(p)\top}\to uu^{\!\top}$; both statements are independent of the
signs of $u^{(p)}$ and $u$.
\end{enumerate}
\end{lemma}
\begin{proof}
\emph{(i)} Weyl's inequality gives $|\operatorname{eig}_i(A^{(p)})-\operatorname{eig}_i(A)|\le\|A^{(p)}-A\|\to0$,
so the eigenvalue $\zeta^{(p)}$ nearest $\zeta$ is, for all large enough $p$,  the unique one
within half the spectral gap of $\zeta$, is simple, and $\zeta^{(p)}\to\zeta$.

To show $c_p := |\langle v^{(p)}, v\rangle| \to 1$, it suffices to establish that every subsequence $c_{p_k}$ has a further subsequence converging to 1.

Let $p_k$ define a subsequence $c_{p_k} = |\langle v^{(p_k)}, v\rangle| $. Since the $v^{(p_k)}$ live on the compact unit sphere, there is a further subsequence $\{p(j)\}$ of $\{p_k\}$ such that $v^{(p(j))}$ converges to a limiting unit vector $\bar v$.  Hence
$$
A \bar v = \lim_{j \to \infty} A^{(p(j))} v^{(p(j))} = \lim_{j \to \infty} \zeta^{(p(j))} v^{(p(j))} = \zeta \bar v.
$$
By simplicity of the eigenvalue $\zeta$, this means $\bar v = \pm v$. Hence
$$
c_{p(j)} = |\langle v^{(p(j))},  v \rangle | = |\langle v^{(p(j))},  \bar v \rangle | \to 1
$$
as desired.

\emph{(ii)} Expanding with $\operatorname{tr}(uu^{\!\top})=\operatorname{tr}(vv^{\!\top})=1$ and
$\operatorname{tr}(uu^{\!\top}vv^{\!\top})=\langle u,v\rangle^2$,
\[
\|uu^{\!\top}-vv^{\!\top}\|_{\mathrm{Fr}}^2
=\operatorname{tr}\big((uu^{\!\top}-vv^{\!\top})^2\big)
=2-2\langle u,v\rangle^2=2\sin^2\!\angle(u,v),
\]
by \eqref{eq:angledef}. The equivalence follows since all norms on
$\R^{m\times m}$ are equivalent.

\end{proof}
\medskip

\subsection{Principal coordinates in the loading frame}

The following lemma describes the change of variables from the loadings $B$ and
factor returns $f$ of \eqref{eq:lfm} to the principal coordinates $b$ and
$\phi$ of Section~\ref{sec:fac-mod}.
\medskip

\begin{lemma}[Change of basis]\label{lem:basis}
    Given the $p \times k$  rank $k$ matrices $B$ and $b$ and other notation defined in Section \ref{sec:fac-mod},
\begin{enumerate}
    \item $C = (B^\top B)^{-1}B^\top b$ is the unique invertible $k \times k$ matrix such that $b = BC$;
    \item $C^{-1} = b^\top B$ and hence
    \begin{equation}
        \phi := C^{-1}f = b^\top Bf ;
    \end{equation}
    and
    \item $\E[\phi \phi^\top] = \Delta_0$.
\end{enumerate}
\end{lemma}

\begin{proof}
    For part 1, if $BC = b$, then $B^\top BC = B^\top b$, and left-multiplication by $(B^\top B)^{-1}$ gives the result.

    For part 2,
    \begin{equation}
        b^\top BC = b^\top b = I_k \text{ and } C(b^\top B) = (B^\top B)^{-1} B^\top b b^\top B.
    \end{equation}
Since $bb^\top = \Pib$, the orthogonal projector onto $\col(B)$, we have $bb^\top B = B$, and so $C(b^\top B) = I_k = (b^\top B)C$.

For part 3,
\begin{equation}
    \E[\phi \phi^\top] = b^\top B \Sigma_f B^\top b =  b^\top \Sigma_0 b = b^\top b \Delta_0 b^\top b = \Delta_0.
\end{equation}
\end{proof}
\medskip

The following proposition records the $\Sigma_f$-free loading frame $\bar b$ of Section~\ref{sec:main-proof}, the representation in it of the principal frame $b$ of Section~\ref{sec:fac-mod}, and the limits of both as $p\to\infty$.
\medskip

\begin{proposition}[Principal coordinates in the loading frame]\label{prop:chart}
Let $\bar b=B(B^{\!\top}B)^{-1/2}$, $G_B^{(p)}=B^{\!\top}B/p$, $M_f^{(p)}=(G_B^{(p)})^{1/2}\Sigma_f(G_B^{(p)})^{1/2}$ and $M_f=G_B^{1/2}\Sigma_fG_B^{1/2}$ be as in \eqref{eq:bbar}, \eqref{eq:Mpdef} and Assumption~\ref{asm:sep}, and write $\mu_1^{(p)}\ge\cdots\ge\mu_k^{(p)}$ for the eigenvalues of $M_f^{(p)}$.
\begin{enumerate}
\item[(i)] For every $p$, $\bar b^{\!\top}\bar b=I_k$, $\bar b\bar b^{\!\top}=\Pib$, $B=\sqrt p\,\bar b\,(G_B^{(p)})^{1/2}$, and $\Sigma_0/p=\bar b\,M_f^{(p)}\,\bar b^{\!\top}$.  Consequently the nonzero eigenvalues of $\Sigma_0/p$, counted with multiplicity, are $\mu_1^{(p)},\dots,\mu_k^{(p)}$, all positive; these are the $\mu_j^{(p)}$ of Section~\ref{sec:fac-mod}.
\item[(ii)] As $p\to\infty$, $(G_B^{(p)})^{1/2}\to G_B^{1/2}$, $M_f^{(p)}\to M_f$, and $\mu_j^{(p)}\to\mu_j$.  Hence, for large enough $p$, $\mu_1^{(p)}>\cdots>\mu_k^{(p)}>0$, and the spectral decomposition of $\Sigma_0=B\Sigma_fB^{\!\top}$ is $b\,\Delta_0\,b^{\!\top}$ with $\Delta_0:=p\operatorname{diag}(\mu_1^{(p)},\dots,\mu_k^{(p)})$, unique up to signs of the columns of $b$.
\item[(iii)] For $p$ as in (ii), $\beta_j^{(p)}:=\bar b^{\!\top}b_j$ is the $j$th unit eigenvector of $M_f^{(p)}$; equivalently, $b_j=\bar b\,\beta_j^{(p)}$.  Moreover $bb^{\!\top}=\bar b\bar b^{\!\top}$.
\item[(iv)] $\angle(\beta_j^{(p)},\beta_j)\to0$, where $\beta_j$ is the $j$th unit eigenvector of $M_f$ (unique up to sign by Assumption~\ref{asm:sep}).
\end{enumerate}
\end{proposition}

\begin{proof}
\emph{(i)} $\bar b^\top\bar b=(B^\top B)^{-1/2}B^\top B(B^\top B)^{-1/2}=I_k$ and $\col(\bar b)=\col(B)=\mathcal B$, so $\bar b\bar b^\top$ is the orthogonal projector onto $\mathcal B$. Since $(B^\top B)^{1/2}=\sqrt p\,(G_B^{(p)})^{1/2}$, we have $B=\bar b(B^\top B)^{1/2}=\sqrt p\,\bar b\,(G_B^{(p)})^{1/2}$, and substituting into $\Sigma_0=B\Sigma_fB^\top$ gives $\Sigma_0/p=\bar b\,M_f^{(p)}\,\bar b^\top$.
The matrix $M_f^{(p)}$ is positive definite, because $\Sigma_f$ is and $G_B^{(p)}$ is ($B$ has rank $k$). Write $M_f^{(p)}=\sum_{j=1}^k\mu_j^{(p)}v_jv_j^\top$ with $v_1,\dots,v_k$ orthonormal. Then
$$\frac{\Sigma_0}{p}=\sum_{j=1}^k\mu_j^{(p)}\,(\bar b v_j)(\bar b v_j)^\top ,$$
and since $\bar b$ is an isometry the vectors $\bar b v_1,\dots,\bar b v_k$ are orthonormal. This is therefore a spectral decomposition of $\Sigma_0/p$, whose nonzero eigenvalues are the positive numbers $\mu_j^{(p)}$, with unit eigenvectors $\bar b v_j$.

\emph{(ii)} Since $G_B^{(p)}\to G_B$ (Assumption~\ref{asm:gram}) with $G_B$ positive definite, and the positive definite square root is continuous on positive definite matrices (a standard fact), $(G_B^{(p)})^{1/2}\to G_B^{1/2}$ and hence $M_f^{(p)}\to M_f$. Weyl's inequality, as in the proof of Lemma~\ref{lem:econv}(i), gives $\mu_j^{(p)}\to\mu_j$ for each $j$. The $\mu_j$ are distinct and positive (Assumption~\ref{asm:sep}), so for large enough $p$ so are the $\mu_j^{(p)}$. By (i), $\Sigma_0/p$ then has the simple nonzero eigenvalues $\mu_1^{(p)}>\cdots>\mu_k^{(p)}$ with unit eigenvectors $\bar b v_j$, which gives $\Sigma_0=b\Delta_0b^\top$ with $b_j=\pm\bar b v_j$, unique up to sign.

\emph{(iii)} By (ii), $b_j=\pm\bar b v_j$, so $\beta_j^{(p)}=\bar b^\top b_j=\pm v_j$ by (i), which is a unit eigenvector of $M_f^{(p)}$ for its $j$th eigenvalue $\mu_j^{(p)}$; and $\bar b\,\beta_j^{(p)}=\pm\bar b v_j=b_j$. Finally, since $v_1,\dots,v_k$ is an orthonormal basis of $\R^k$, $bb^\top=\sum_j\bar b v_jv_j^\top\bar b^\top=\bar b\bar b^\top$.

\emph{(iv)} By (ii), $M_f^{(p)}\to M_f$, whose eigenvalues are simple (Assumption~\ref{asm:sep}). By (iii), $\beta_j^{(p)}$ is a unit eigenvector of $M_f^{(p)}$ for $\mu_j^{(p)}$, the eigenvalue converging to $\mu_j$, so Lemma~\ref{lem:econv}(i) gives $\angle(\beta_j^{(p)},\beta_j)\to0$.
\end{proof}
\medskip

\subsection{The observable dual Gram matrix in the limit}

The following lemma supplies parts (a) and (b) of the proposition that follows.
\medskip

\begin{lemma}[Specific return concentration]\label{lem:zconc}
Let Assumptions~\ref{asm:noise} and~\ref{asm:delta} hold, write
$Z_{\cdot l}=z^{(p,l)}$ for the $l$th column of $Z=Z^{(p)}$, and let
$a=a^{(p)}\in\R^p$ be deterministic unit vectors.  Conditional on $F$,
almost surely as $p\to\infty$:
\begin{enumerate}
\item[(i)] $a^{\!\top}Z_{\cdot l}/\sqrt p\to0$ for each $l$;
\item[(ii)] $\|Z_{\cdot l}\|^2/p\to\delta^2$ for each $l$;
\item[(iii)] $Z_{\cdot l}^{\!\top}Z_{\cdot m}/p\to0$ for each $l\ne m$.
\end{enumerate}
\end{lemma}

\begin{proof}
Fix $F\in\R^{k\times n}$, and write $\Pr$ and $\E$ for probability and
expectation under $Q(F,\cdot)$; all constants below are deterministic.  Each part follows from a tail bound and the
Borel--Cantelli lemma, applied for each $\epsilon$ in a sequence
decreasing to $0$.

\emph{(i)} By Markov's inequality and condition~(i) of
Assumption~\ref{asm:noise},
\[
  \Pr\big(|a^{\!\top}Z_{\cdot l}|>\epsilon\sqrt p\big)
  \le\frac{\E(a^{\!\top}Z_{\cdot l})^4}{\epsilon^4p^2}
  \le\frac{C}{\epsilon^4p^2},
\]
which is summable in $p$.

\emph{(ii)} Fix $l$ and set $S_p:=\|Z_{\cdot l}\|^2=\sum_{i\le p}Z_{il}^2$,
so that $\E S_p=\tr\Delta^{(p)}$.  Chebyshev's inequality and
condition~(ii) give
\[
  \Pr\big(|S_p-\tr\Delta^{(p)}|>\epsilon p\big)\le\frac{C}{\epsilon^2p}.
\]
This bound is not summable over all $p$, but it is summable along
$p_k:=k^2$.  Hence $(S_{p_k}-\tr\Delta^{(p_k)})/p_k\to0$, and therefore
$S_{p_k}/p_k\to\delta^2$ by Assumption~\ref{asm:delta}, almost surely.
By the nesting convention the summands $Z_{il}^2$ do not depend on $p$, so
$S_p$ is nondecreasing in $p$.  For $p_k\le p<p_{k+1}$,
\[
  \frac{p_k}{p_{k+1}}\cdot\frac{S_{p_k}}{p_k}
  \;\le\;\frac{S_p}{p}\;\le\;
  \frac{p_{k+1}}{p_k}\cdot\frac{S_{p_{k+1}}}{p_{k+1}},
\]
and $p_{k+1}/p_k\to1$, so both bounds converge to $\delta^2$ almost surely.

\emph{(iii)} Fix $l\ne m$ and write $u:=Z_{\cdot l}$, $v:=Z_{\cdot m}$.
Given $F$, $u$ and $v$ are independent, so we may integrate first over $u$
with $v$ held fixed; applying Assumption~\ref{asm:noise}(i) in the direction $v/\|v\|$ on
$\{v\ne0\}$, and then Assumption~\ref{asm:noise}(ii) together with
$\tr\Delta^{(p)}\le p\|\Delta^{(p)}\|\le p\sqrt C$ from \eqref{eq:Deltabd},
\[
  \E\big(u^{\!\top}v\big)^4\le C\,\E\|v\|^4
  =C\Big(\Var\big(\|v\|^2\big)+\big(\tr\Delta^{(p)}\big)^2\Big)
  \le C\big(Cp+Cp^2\big)\le 2C^2p^2 .
\]
Markov's inequality gives
$\Pr\big(|u^{\!\top}v|>\epsilon p\big)\le 2C^2/(\epsilon^4p^2)$, which is
summable in $p$.
\end{proof}

Only part~(ii) uses the nesting convention, through the monotonicity of
$S_p$; the tail bounds in parts~(i) and~(iii) are summable over all $p$.
\medskip

\begin{proposition}[Limit of the observable dual $W^{(p)}$]\label{prop:dual}
Under our Assumptions,
conditional on $F$ and a.s.\ as $p\to\infty$:
\begin{enumerate}
\item[(a)] $Z^{\!\top}Z/(np)\to(\delta^2/n)I_n$ in spectral norm;
\item[(b)] $\|\Pib Z\|_{\mathrm{Fr}}=\|b^{\!\top}Z\|_{\mathrm{Fr}}=o(\sqrt p)$;
\item[(c)] $W^{(p)}\to W$ of \eqref{eq:Wlim} in spectral norm, with top
$k$ eigenpairs $(\lambda_j+\delta^2/n,\,w_j)$; and
\item[(d)] $\theta_j^{(p)}\to\lambda_j+\delta^2/n>0$ and
$\angle\big(w_j^{(p)},w_j\big)\to0$, $j=1,\dots,k$.
\end{enumerate}
\end{proposition}
\begin{proof}
\emph{(a)} As $n$ is fixed, entrywise a.s.\ convergence implies
spectral-norm convergence.  The $(l,l)$ entry of $Z^{\!\top}Z/(np)$ is
$\|Z_{\cdot l}\|^2/(np)\to\delta^2/n$ by Lemma~\ref{lem:zconc}(ii), and the
$(l,m)$ entry, $l\ne m$, is $Z_{\cdot l}^{\!\top}Z_{\cdot m}/(np)\to0$ by
Lemma~\ref{lem:zconc}(iii).  Hence the matrix limit is
$(\delta^2/n)I_n$.

\emph{(b)} The $(j,m)$ entry of $b^{\!\top}Z$ is $b_j^{\!\top}Z_{\cdot m}$
with $b_j$ a deterministic unit vector, so $b_j^{\!\top}Z_{\cdot m}=o(\sqrt p)$
a.s.\ by Lemma~\ref{lem:zconc}(i); summing the squares of the $kn$ entries gives
$\|b^{\!\top}Z\|_{\mathrm{Fr}}=o(\sqrt p)$.  Since $b$ has orthonormal columns,
$\|\Pib Z\|_{\mathrm{Fr}}=\|bb^{\!\top}Z\|_{\mathrm{Fr}}=\|b^{\!\top}Z\|_{\mathrm{Fr}}$.

\emph{(c)} Substituting $Y=b\Phi+Z$ and using $b^{\!\top}b=I$,
\begin{equation}
  W^{(p)}=\frac{\Phi^{\!\top}\Phi}{np}
   +\frac{\Phi^{\!\top}b^{\!\top}Z+Z^{\!\top}b\Phi}{np}
   +\frac{Z^{\!\top}Z}{np}.
\end{equation}
For the systematic block, $\Phi=b^{\!\top}BF$ by Lemma~\ref{lem:basis}, and $bb^{\!\top}B=B$ because $bb^{\!\top}$ is the orthogonal projector onto $\col(b)=\col(B)$ \eqref{eq:borth}.  Hence $\Phi^{\!\top}\Phi=F^{\!\top}B^{\!\top}bb^{\!\top}BF=F^{\!\top}B^{\!\top}BF$, and Assumption~\ref{asm:gram} gives
$$\Phi^{\!\top}\Phi/(np)=F^{\!\top}G_B^{(p)}F/n=W_0^{(p)}\to F^{\!\top}G_BF/n=W_0.$$
In particular $\|\Phi\|^2=\|\Phi^{\!\top}\Phi\|=np\,\|W_0^{(p)}\|=O(p)$, since $W_0^{(p)}$ converges; and by part (b), $\|b^\top Z\|_{\mathrm{Fr}} = o(\sqrt{p})$.  Hence
the cross block is bounded in norm by
$$2\|\Phi\|\,\|b^{\!\top}Z\|_{\mathrm{Fr}}/(np)\le 2\,O(\sqrt p)\,o(\sqrt p)/(np)=o(1).$$
The noise
block converges to $(\delta^2/n)I_n$ by part (a), and hence $W^{(p)}\to W$.  The eigenstructure of
$W$ is read off in \eqref{eq:Wlim}: $W_0$ has rank $k$ with eigenpairs
$(\lambda_j,w_j)$ (Assumption~\ref{asm:reg} and Lemma~\ref{lem:gramdual}(i); see Section~\ref{sec:gram}), and the scalar shift
$(\delta^2/n)I_n$ moves eigenvalues by $\delta^2/n$ while fixing
eigenvectors.

\emph{(d)} The top $k$ eigenvalues of $W$ are simple
(Assumption~\ref{asm:reg} and Lemma~\ref{lem:gramdual}(i)); apply Lemma~\ref{lem:econv}(i) on the a.s.\ event
of (c).  The limits are positive since $\lambda_j+\delta^2/n\ge
\lambda_k+\delta^2/n>0$.
\end{proof}

\subsection{Sufficient conditions for Assumption~\ref{asm:noise}}
\label{subsec:crosscorr}

Assumption~\ref{asm:noise} is stated in terms of moments.  This
subsection gives three primitive conditions under which it holds; the
third accommodates cross-sectionally correlated specific returns driven by
heavy-tailed innovations.  We then comment on what the assumption excludes,
and on specific covariance matrices whose largest eigenvalue grows with $p$.

Throughout, $\|\cdot\|$ is the spectral norm, and for an infinite array
$\Sigma$, $\Sigma_{[p]}$ denotes its leading $p\times p$ block.  For a
symmetric matrix $S$,
\begin{equation}\label{eq:rowsum}
  \|S\|\le\max_i\sum_j|S_{ij}|,
\end{equation}
because the spectral radius is bounded by the induced $\ell^\infty$ norm.

\medskip
\begin{proposition}[Sufficient conditions for Assumption~\ref{asm:noise}]
\label{prop:suff}
Each of the following conditions on the specific returns array
$\{Z_{il}: i\ge1,\ 1\le l\le n\}$ implies Assumption~\ref{asm:noise}.
\begin{enumerate}
\item[(a)] \emph{Independent entries.}  Conditional on $F$, the entries
  $Z_{il}$ are mutually independent and mean zero, and stationary across
  observations: for each $i$, the variables $Z_{i1},\dots,Z_{in}$ have the
  same conditional law given $F$.  The conditional variances
  $\Var(Z_{il}\mid F)=\delta_i^2$ do not depend on $F$, and
  $\sup_i\E[Z_{il}^4\mid F]\le\kappa_4<\infty$.  In this case
  $\Delta^{(p)}=\operatorname{diag}(\delta_1^2,\dots,\delta_p^2)$, and
  Assumption~\ref{asm:noise} conditions~(i) and~(ii) hold with $C=4\kappa_4$.
\item[(b)] \emph{Gaussian specific returns.}  Conditional on $F$, the
  columns $(Z_{1l},Z_{2l},\dots)$, $l=1,\dots,n$, are independent and
  identically distributed mean-zero Gaussian sequences whose common
  covariance array $\Sigma$ is deterministic, with
  $\|\Sigma_{[p]}\|\le C_\Delta$ for all $p$.  In this case
  $\Delta^{(p)}=\Sigma_{[p]}$, and conditions~(i) and~(ii) hold with
  $C=3C_\Delta^2$.
\item[(c)] \emph{Linear representation.}  There are an infinite array
  $\{\varepsilon_{ml}: m\ge1,\ 1\le l\le n\}$ of random variables
  (innovations) and a deterministic infinite real array
  $A=\{A_{im}\}_{i,m\ge1}$, each row of which has finitely many nonzero
  entries, such that
  \[
    Z_{il}=\sum_{m\ge1}A_{im}\,\varepsilon_{ml},
    \qquad i\ge1,\ 1\le l\le n .
  \]
  Conditional on $F$, the innovations are mutually independent, with
  \[
    \E[\varepsilon_{ml}\mid F]=0,\qquad
    \E[\varepsilon_{ml}^2\mid F]=1,\qquad
    \sup_{m,l}\E\big[\varepsilon_{ml}^{4}\mid F\big]\le\nu_{4}<\infty
    \ \text{a.s.},
  \]
  and stationary across observations: for each $m$, the variables
  $\varepsilon_{m1},\dots,\varepsilon_{mn}$ have the same conditional law
  given $F$ (the laws may differ across $m$).  Let $c(p)$ be the largest
  $m$ such that $A_{im}\ne0$ for some $i\le p$, let $A_{[p]}$ denote the
  $p\times c(p)$ upper-left block of $A$, and suppose
  $\|A_{[p]}A_{[p]}^{\!\top}\|\le C_\Delta$ for all $p$.  In this case
  $\Delta^{(p)}=A_{[p]}A_{[p]}^{\!\top}$, and conditions~(i) and~(ii) hold
  with $C=(\nu_4+3)C_\Delta^2$.
\end{enumerate}
\end{proposition}

\begin{proof}
In each case fix a column $l$, write $z:=z^{(p,l)}$, and work
conditionally on $F$; all constants are deterministic.  Since the bounds
obtained do not depend on $l$, they hold for all $l$.

\emph{(a)} The columns are conditionally independent because all entries
are, and identically distributed by stationarity; they have mean zero and
conditional covariance $\operatorname{diag}(\delta_1^2,\dots,\delta_p^2)$,
which does not depend on $F$.  Jensen's inequality gives
$\delta_i^2\le\sqrt{\kappa_4}$.  For a unit vector $a$, expanding the
fourth power of $a^{\!\top}z=\sum_ia_iZ_{il}$, only terms in which every
index appears an even number of times survive, by independence and zero
means, so
\[
  \E(a^{\!\top}z)^4=\sum_ia_i^4\E Z_{il}^4
   +3\sum_{i\ne i'}a_i^2a_{i'}^2\delta_i^2\delta_{i'}^2
  \le\kappa_4\sum_ia_i^4+3\Big(\sum_ia_i^2\delta_i^2\Big)^2
  \le\kappa_4+3\kappa_4 ,
\]
using $\sum_ia_i^4\le(\sum_ia_i^2)^2=1$.  By independence,
$\Var(\|z\|^2)=\sum_{i\le p}\Var(Z_{il}^2)\le\sum_{i\le p}\E Z_{il}^4\le\kappa_4p$.

\emph{(b)} The columns are conditionally i.i.d.\ with mean zero and
conditional covariance $\Sigma_{[p]}$, which does not depend on $F$.  For
a unit vector $a$, $a^{\!\top}z$ is conditionally
$N(0,a^{\!\top}\Sigma_{[p]}a)$, so
$\E(a^{\!\top}z)^4=3(a^{\!\top}\Sigma_{[p]}a)^2\le3C_\Delta^2$.  Writing
$\Sigma_{[p]}=U\Lambda U^{\!\top}$ with $U$ orthogonal, $z$ has the law of
$U\Lambda^{1/2}g$ with $g\sim N(0,I_p)$, so $\|z\|^2$ has the law of
$\sum_i\Lambda_{ii}g_i^2$ and
$\Var(\|z\|^2)=2\tr\Sigma_{[p]}^2\le2\|\Sigma_{[p]}\|\tr\Sigma_{[p]}\le2C_\Delta^2p$.

\emph{(c)} Write
$\varepsilon^{(c,l)}:=(\varepsilon_{1l},\dots,\varepsilon_{cl})^{\!\top}$.
Since rows $i\le p$ of $A$ vanish beyond column $c(p)$,
$z^{(p,l)}=A_{[p]}\varepsilon^{(c(p),l)}$, and the first $p$ coordinates of
$z^{(p',l)}$ coincide with $z^{(p,l)}$ for every $p'\ge p$, consistent with
the nesting convention.  Each column is the same fixed linear map applied
to the innovation vector $(\varepsilon_{1l},\varepsilon_{2l},\dots)$, and
these vectors are conditionally independent and identically distributed
across $l$ given $F$; hence the columns are conditionally i.i.d., with
mean zero and conditional covariance $A_{[p]}A_{[p]}^{\!\top}=\Delta^{(p)}$,
which does not depend on $F$.

For the moment bounds we use two identities.  Let
$\xi_1,\dots,\xi_c$ be independent, mean-zero, unit-variance random
variables with $\max_m\E\xi_m^4\le\nu_4$ (so $\nu_4\ge1$).  For every
deterministic $u\in\R^c$ and every deterministic symmetric $c\times c$
matrix $M$,
\begin{align}
  \E\big(u^{\!\top}\xi\big)^4
  &=\sum_m u_m^4\,\E\xi_m^4+3\sum_{m\ne m'}u_m^2u_{m'}^2
   \;\le\;(\nu_4+3)\,\|u\|^4, \label{eq:M1}\\
  \E\big(\xi^{\!\top}M\xi-\tr M\big)^2
  &=\sum_m M_{mm}^2\Var\big(\xi_m^2\big)+2\sum_{m\ne m'}M_{mm'}^2
   \;\le\;(\nu_4+1)\tr M^2. \label{eq:M2}
\end{align}
Both follow by expanding and using independence, zero means and unit
variances.  In~\eqref{eq:M1} only terms in which every index appears an
even number of times survive.  For~\eqref{eq:M2}, write
$\xi^{\!\top}M\xi-\tr M=\sum_mM_{mm}(\xi_m^2-1)+\sum_{m\ne m'}M_{mm'}\xi_m\xi_{m'}$;
the two sums are uncorrelated, and $\E[\xi_m\xi_{m'}\xi_r\xi_{r'}]$ with
$m\ne m'$, $r\ne r'$ vanishes unless $\{r,r'\}=\{m,m'\}$.  The
inequality uses $\Var(\xi_m^2)\le\nu_4-1$.

Now apply these with $\xi:=\varepsilon^{(c(p),l)}$.  For a unit vector $a$,
$a^{\!\top}z=u^{\!\top}\xi$ with $u:=A_{[p]}^{\!\top}a$ and
$\|u\|^2=a^{\!\top}\Delta^{(p)}a\le C_\Delta$, so~\eqref{eq:M1} gives
$\E(a^{\!\top}z)^4\le(\nu_4+3)C_\Delta^2$.  Next,
$\|z\|^2=\xi^{\!\top}M\xi$ with $M:=A_{[p]}^{\!\top}A_{[p]}$, which has the
same nonzero eigenvalues as $\Delta^{(p)}$.  Hence
$\E\|z\|^2=\tr M=\tr\Delta^{(p)}$ and
$\tr M^2=\tr(\Delta^{(p)})^2\le\|\Delta^{(p)}\|\tr\Delta^{(p)}\le C_\Delta^2p$,
so~\eqref{eq:M2} gives
$\Var(\|z\|^2)\le(\nu_4+1)C_\Delta^2p\le(\nu_4+3)C_\Delta^2p$.
\end{proof}

\paragraph{Relations among the conditions.}
If the specific variances in part~(a) are bounded below,
$\inf_i\delta_i^2=:\underline\delta^2>0$, then (a) is the special case of
(c) with $A$ diagonal with entries $\delta_i$, innovations
$\varepsilon_{il}=Z_{il}/\delta_i$, $\nu_4=\kappa_4/\underline\delta^4$ and
$C_\Delta=\sqrt{\kappa_4}$.    Similarly, when
the blocks $\Sigma_{[p]}$ are positive definite, (b) can be obtained from
(c) through the Cholesky factorization of $\Sigma$, but the direct
verification above avoids this construction.

\paragraph{Example: spatial AR(1) correlation with heavy-tailed innovations.}
Order the assets, let $|\rho|<1$ and $0<\delta_i^2\le\bar\delta^2$, and let
the innovations $\varepsilon_1,\varepsilon_2,\dots$ at a given date be
independent, mean zero and of unit variance, with uniformly bounded
fourth moments, and with the same laws at every date; standardized $t_\nu$
innovations with $\nu>4$ are an example.  Define $Z_i:=\delta_iU_i$, where
\[
  U_1=\varepsilon_1,\qquad
  U_i=\rho\,U_{i-1}+\sqrt{1-\rho^2}\,\varepsilon_i\quad(i\ge2).
\]
Each $Z_i$ is a finite linear combination of
$\varepsilon_1,\dots,\varepsilon_i$, so the representation of
Proposition~\ref{prop:suff}(c) holds with $c(p)=p$.  Each $U_i$ has unit
variance and $\operatorname{Cov}(U_i,U_j)=\rho^{|i-j|}$, whatever the laws
of the innovations.  Hence $\Delta^{(p)}=\Sigma_{[p]}$ with
$\Sigma_{ij}=\delta_i\delta_j\,\rho^{|i-j|}$, that is, $\Sigma=DRD$ with
$D=\operatorname{diag}(\delta_i)$ and $R_{ij}=\rho^{|i-j|}$.
By~\eqref{eq:rowsum},
\[
  \|\Sigma_{[p]}\|\le\bar\delta^2\,\|R_{[p]}\|
  \le\bar\delta^2\Big(1+2\sum_{h\ge1}|\rho|^h\Big)
  =\bar\delta^2\,\frac{1+|\rho|}{1-|\rho|},
\]
uniformly in $p$, so Proposition~\ref{prop:suff}(c) applies.
Assumption~\ref{asm:delta} holds if
$\frac1p\sum_{i\le p}\delta_i^2\to\delta^2$.

\medskip
\begin{remark}[What condition~(ii) excludes]
Condition~(ii) of Assumption~\ref{asm:noise} is not implied by
condition~(i).  Let the array $\{U_{il}\}$ satisfy
Proposition~\ref{prop:suff}(a), with columns $u^{(l)}$ and
$\frac1p\sum_{i\le p}\delta_i^2\to\delta^2$, and let $s_1,\dots,s_n$ be
i.i.d.\ positive random variables, independent of $F$ and of $\{U_{il}\}$,
with $\E s_l^2=1$, $\E s_l^4<\infty$ and $\Var(s_l^2)>0$.  Set
$z^{(l)}:=s_l\,u^{(l)}$, a date-specific volatility shock common to all
assets.  Conditional on $F$ the columns $z^{(l)}$ are i.i.d.\ with mean
zero and the $F$-free covariance
$\operatorname{diag}(\delta_1^2,\dots,\delta_p^2)$, and condition~(i)
holds, since $\E(a^{\!\top}z^{(l)})^4=\E s_l^4\,\E(a^{\!\top}u^{(l)})^4\le4\kappa_4\E s_l^4$.
But by Lemma~\ref{lem:zconc} applied to the $u^{(l)}$, conditionally on
$s_1,\dots,s_n$,
\[
  \frac{Z^{\!\top}Z}{p}\longrightarrow
  \delta^2\operatorname{diag}(s_1^2,\dots,s_n^2)\quad\text{a.s.},
\]
which is not $\delta^2I_n$, so (P1) fails; correspondingly
$\Var(\|z^{(l)}\|^2)\ge\Var(s_l^2)\,(\tr\Delta^{(p)})^2$ grows like $p^2$.
The noise block of the limiting dual Gram matrix is then the diagonal
matrix of per-date levels considered in the discussion of stationarity at the end of this subsection.
\end{remark}

\medskip
\begin{remark}[Growing specific covariance]
Condition~(i) forces $\sup_p\|\Delta^{(p)}\|<\infty$ by
\eqref{eq:Deltabd}.  Suppose instead that conditions~(i) and~(ii) hold
with right-hand sides $Cp^{2\alpha}$ and $Cp^{1+\alpha}$ respectively, for
some $\alpha\in[0,1)$; by the computations in the proof of
Proposition~\ref{prop:suff}, this is the case under (b) or (c) when
$\|\Delta^{(p)}\|\le C_\Delta p^{\alpha}$.  In the proof of
Lemma~\ref{lem:zconc}, part~(ii) goes through for every $\alpha<1$: the
Chebyshev bound becomes $Cp^{\alpha-1}/\epsilon^2$, which is summable along
$p_k:=\lceil k^r\rceil$ with $r(1-\alpha)>1$, and still
$p_{k+1}/p_k\to1$.  Parts~(i) and~(iii) yield tail bounds of order
$p^{2\alpha-2}$ (in part~(iii), $\E\|v\|^4=O(p^2)$ still holds because
$\tr\Delta^{(p)}/p\to\delta^2$), and these are summable when
$\alpha<1/2$.  Hence for $\alpha<1/2$ all of our results hold unchanged.
For $1/2\le\alpha<1$ the same argument gives parts~(i) and~(iii) only in
probability.  For part~(i), bounded conditional moments of order $2q$ with
$q(1-\alpha)>1$ restore summability by the same argument; we expect, but do
not prove, that the almost sure conclusions persist for every $\alpha<1$
under suitable higher-order moment conditions.
\end{remark}

\paragraph{Relaxing stationarity.}
Stationarity of the specific returns is stronger than our arguments
need: the proofs use it only through the date-invariance of the limit
$\delta^2$.  More generally, the proofs of Lemma~\ref{lem:zconc} and
Proposition~\ref{prop:dual} go through if the columns $z^{(l)}$ are
conditionally independent given $F$ but not identically distributed,
with $F$-free conditional covariances $\Delta^{(p,l)}$ and with
conditions~(i) and~(ii) holding uniformly in $l$, provided
$\tr\Delta^{(p,l)}/p\to\delta^2$ for every $l$ with the same limit
$\delta^2$.  (The observation covariance
$\Sigma_y^{(l)}=\Sigma_0+\Delta^{(p,l)}$ then depends on $l$, but the
systematic covariance $\Sigma_0$ and the targets $b_j$ do not.)  We
conjecture that the requirement of a common limit may also be relaxed.
In that case, however, the noise block of the limiting dual Gram matrix
would be the diagonal matrix of the per-date limits rather than a
multiple of $I_n$, and the estimator of Theorem~\ref{thm:obsfloor} would
need to be modified.  We do not pursue this here.

\subsection{A consequence of the model assumptions}

{ The following lemma is included to help illuminate the consequences of our Assumptions \ref{asm:fm} and \ref{asm:noise}, but is not used in the proofs.
\medskip

}

\begin{lemma}[All-pairs uncorrelatedness and the observation covariance]
\label{lem:uncorr}
Suppose Assumptions~\ref{asm:fm} and~\ref{asm:noise} hold, and write
$\mathcal G:=\sigma(F)$ for the $\sigma$-algebra generated by the factor
path.  Then:
\begin{enumerate}
  \item For all $i\ge1$, $a\in\{1,\dots,k\}$ and all $l,m\in\{1,\dots,n\}$
    the product $Z_{il}f_a^{(m)}$ is integrable and
    \begin{equation}\label{eq:allpairs}
      \E\big[Z_{il}f_a^{(m)}\big]=0,
      \qquad\text{equivalently}\qquad
      \E\big[z^{(l)}f^{(m)\top}\big]=0\in\R^{p\times k}.
    \end{equation}
  \item For each $l$, the observation $y^{(l)}=Bf^{(l)}+z^{(l)}$ has
    \begin{equation}\label{eq:Syl}
      \Sigma_y^{(l)}:=\E\big[y^{(l)}y^{(l)\top}\big]
        =B\Sigma_fB^{\!\top}+\Delta_z,
      \qquad
      \Delta_z=\Delta^{(p)},
    \end{equation}
    the conditional covariance of Assumption~\ref{asm:noise}.
    In particular $\Sigma_y^{(l)}$ does not depend on $l$ and equals the
    matrix $\Sigma_y$ of \eqref{eq:cov}.
\end{enumerate}
\end{lemma}

\begin{proof}
(1)  Conditional mean zero identifies the conditional second moment with
the conditional variance, so $\E[Z_{il}^2\mid\mathcal G]=\delta_i^2$,
and $\delta_i^2\le C^{1/2}$ by \eqref{eq:Deltabd}; taking
expectations, $\E[Z_{il}^2]\le C^{1/2}<\infty$.  By
Assumption~\ref{asm:fm}, $\E[(f_a^{(m)})^2]\le\E\|f^{(m)}\|^2<\infty$.
Cauchy--Schwarz then gives
\[
  \E\big|Z_{il}f_a^{(m)}\big|
  \le\big(\E[Z_{il}^2]\big)^{1/2}\big(\E[(f_a^{(m)})^2]\big)^{1/2}
  \le C^{1/4}\big(\E\|f^{(m)}\|^2\big)^{1/2}<\infty,
\]
so the product is integrable. Since
$f_a^{(m)}$ is $\mathcal G$-measurable, we have
$\E[Z_{il}f_a^{(m)}\mid\mathcal G]=f_a^{(m)}\,\E[Z_{il}\mid\mathcal G]=0$
almost surely.  Taking expectations gives \eqref{eq:allpairs}.

(2)  Expanding $y^{(l)}y^{(l)\top}$ and taking expectations,
\[
  \Sigma_y^{(l)}=B\,\E\big[f^{(l)}f^{(l)\top}\big]B^{\!\top}
    +B\,\E\big[f^{(l)}z^{(l)\top}\big]
    +\E\big[z^{(l)}f^{(l)\top}\big]B^{\!\top}
    +\E\big[z^{(l)}z^{(l)\top}\big].
\]
The first term is $B\Sigma_fB^{\!\top}$ by Assumption~\ref{asm:fm}, and the
two cross terms vanish by \eqref{eq:allpairs} with $m=l$.  For the last
term, Assumption~\ref{asm:noise} gives
$\E[z^{(l)}z^{(l)\top}\mid\mathcal G]=\Delta^{(p)}$, a constant; taking
expectations yields $\E[z^{(l)}z^{(l)\top}]=\Delta^{(p)}$, which does not
depend on $l$.
\end{proof}

\begin{remark}
Conditioning on $\mathcal G$ rather than on
$\mathcal G_l:=\sigma(f^{(1)},\dots,f^{(l)})\subset\mathcal G$ is what
yields the off-diagonal pairs $l\ne m$ in \eqref{eq:allpairs}; the
contemporaneous condition $\E[Z_{il}\mid\mathcal G_l]=0$ delivers only
$m\le l$.  Conversely, \eqref{eq:allpairs} does not recover
Assumption~\ref{asm:noise}, since it is a second-moment statement whereas
the assumption restricts the conditional law.
\end{remark}

\subsection{The large-$n$ limit}\label{subsec:nlimit}

This subsection records the behavior of the realized quantities $M_F$, $\lambda_j$ and $\omega_j$ of Section~\ref{sec:gram} as the number of observations $n$ grows.
\medskip

\begin{proposition}[Large-$n$ limit of $M_F$] \label{prop:n-limit}
    Drop the fixed-$F$
conditioning of Section~\ref{sec:data}, and suppose   the factor sample
$f^{(1)},f^{(2)},\dots$  satisfies
$$FF^{\!\top}/n\to\Sigma_f \text{ a.s.\ as } n\to\infty.$$
Then, with $M_F$, $M_f$, $\omega_j$ and $\beta_j$ as in \eqref{eq:Mdef},
 $$M_F\to M_f,\qquad \lambda_j\to\mu_j, \qquad \text{and }\;
 \angle(\omega_j,\beta_j)\to0$$
 almost surely.

\end{proposition}
\begin{proof}
Under the hypothesis $FF^{\!\top}/n\to\Sigma_f$ a.s.,
$M_F=G_B^{1/2}(FF^{\!\top}/n)G_B^{1/2}\to G_B^{1/2}\Sigma_fG_B^{1/2}=M_f$.
The eigenvalues $\mu_j$ of $M_f$ are simple (Assumption~\ref{asm:sep}), so
Lemma~\ref{lem:econv}(i), applied to the sequence $M_F$ indexed by $n$, gives
$\lambda_j\to\mu_j$ and $\angle(\omega_j,\beta_j)\to0$.

\end{proof}

\begin{remark}
Proposition~\ref{prop:n-limit} above is not needed for our results, but illuminates the context in which $n$ may vary. It describes a sequential limit, $p \to \infty$, then $n \to \infty$, rather than a joint limit.
We mention two natural sufficient conditions for the hypothesis
$FF^\top/n \to \Sigma_f$.
  Since
$FF^{\!\top}/n=\frac1n\sum_{l=1}^nf^{(l)}f^{(l)\top}$ and $k$ is fixed,
entrywise a.s.\ convergence suffices, so fix an entry $(a,b)$ and write
$X_l := f^{(l)}_af^{(l)}_b$.

\emph{1.}  Suppose the $f^{(l)}$ are independent with
common second-moments $\E[f^{(l)}f^{(l)\top}]=\Sigma_f$ and uniformly
bounded fourth moments $\kappa_f:=\sup_l\E\|f^{(l)}\|^4<\infty$.  Then the
$X_l$ are independent with
common mean $\E X_l=(\Sigma_f)_{ab}$ and, by Cauchy--Schwarz,
\begin{equation*}
  \Var(X_l)\le\E\big[(f^{(l)}_a)^2(f^{(l)}_b)^2\big]
  \le\big(\E[(f^{(l)}_a)^4]\,\E[(f^{(l)}_b)^4]\big)^{1/2}
  \le\kappa_f<\infty ,
\end{equation*}
so $\sum_l\Var(X_l)/l^2<\infty$ and Kolmogorov's strong law for independent
(not necessarily identically distributed) summands \citep{Petrov1975} gives
$\frac1n\sum_{l\le n}X_l\to(\Sigma_f)_{ab}$ a.s.

\emph{2.}  Suppose $f^{(1)},f^{(2)},\dots$ is
stationary and ergodic with $\E\|f^{(1)}\|^2<\infty$ and
$\E[f^{(1)}f^{(1)\top}]=\Sigma_f$.
This covers
serially dependent factor paths, and requires only second moments.
Then it follows that
$(X_l)_{l\ge1}$
is stationary and ergodic with
$\E|X_l|\le\E\|f^{(1)}\|^2<\infty$, and Birkhoff's ergodic theorem gives
$\frac1n\sum_{l\le n}X_l\to\E X_1=(\Sigma_f)_{ab}$ a.s.
\end{remark}

\section{Indeterminacy of the principal frame} \label{sec:orbit}

{
\paragraph{}
In this section we assume the $k \times n$ realized factor returns path $F$ is fixed and we vary the $k \times k$ population factor covariance matrix $\Sigma_f$. Recall that $\beta_j(\Sigma)$ denotes the $j$th unit eigenvector of
$M(\Sigma):=G_B^{1/2}\Sigma G_B^{1/2}$, while the realized direction $\omega_j$, the $j$th unit eigenvector of
\begin{equation}\label{eq:Mhatrecall}
  M_F=G_B^{1/2}\,\frac{FF^{\!\top}}{n}\,G_B^{1/2},
\end{equation}
involves neither $\Sigma$ nor $\Sigma_f$.
Denote by $\mathrm{O}(k)$ the set of $k \times k$ orthogonal matrices.  Lemma~\ref{lem:orbit} below shows that, as $\Sigma$ runs over the admissible
matrices, $M(\Sigma)$ runs over the orthogonal conjugates $O^{\!\top}\Lambda O$, $O\in\mathrm O(k)$, of a fixed diagonal matrix $\Lambda$, so that $\beta_j(\Sigma)$ runs over the whole unit sphere while $\omega_j$ stays fixed; when $k\ge2$ this is what makes
$\angle\big(\omega_j,\beta_j(\Sigma)\big)$ take every value in $[0,\pi/2]$.

This appendix states and proves Lemma~\ref{lem:orbit}, and then proves
Theorem~\ref{thm:rotrange} and Corollary~\ref{cor:noest}.
Write
\begin{equation}\label{eq:Lamdef}
  \Lambda:=\operatorname{diag}(\mu_1,\dots,\mu_k)
\end{equation}
for the diagonal matrix of the population eigenvalues of $M_f=M(\Sigma_f)$, which are held fixed.
  These are the limiting leading eigenvalues of
$\Sigma_0/p$, and are to be distinguished from the realized eigenvalues
$\lambda_j$ of $M_F$ that appear in \eqref{eq:thm}.  Let
\begin{equation}\label{eq:SLam}
  \mathcal S_\Lambda:=\big\{\,\Sigma\in\R^{k\times k}\ \text{positive definite}
   \ :\ M(\Sigma)\ \text{has eigenvalues}\ \mu_1,\dots,\mu_k\,\big\}
\end{equation}
denote the set of admissible factor covariances of
Section~\ref{subsec:nonestimable}, so that $\Sigma_f\in\mathcal S_\Lambda$ by
Assumption~\ref{asm:sep}. Since two symmetric matrices with the same eigenvalues are orthogonally similar, admissibility of $\Sigma$ is equivalent to the requirement that $M(\Sigma)$ is orthogonally similar to $M_f$.
\medskip

\begin{lemma}[Indeterminacy of the population directions]
\label{lem:orbit}
\hfill
\begin{enumerate}
\item[(i)] The admissible set is a single orthogonal orbit:
\begin{equation}\label{eq:orbitlim}
  \mathcal S_\Lambda=\big\{\,\Sigma_O\;:\;O\in \mathrm{O}(k)\,\big\},
  \qquad
  \Sigma_O:=G_B^{-1/2}O^{\!\top}\Lambda\,O\,G_B^{-1/2}.
\end{equation}
\item[(ii)] For every $O\in \mathrm{O}(k)$,
\begin{equation}\label{eq:Morbit}
  M(\Sigma_O)=O^{\!\top}\Lambda\,O,
  \qquad
  \beta_j(\Sigma_O)=\pm\,O^{\!\top}e_j ,
\end{equation}
where $e_j$ is the $j$th standard basis vector of $\R^k$.
\end{enumerate}
\end{lemma}

The two parts together say that prescribing the population spectrum
$\Lambda$ constrains $M(\Sigma)$ only up to orthogonal conjugation: the admissible
$\Sigma$ are indexed by $O\in \mathrm{O}(k)$, and changing $O$ leaves the eigenvalues
$\mu_j$ unchanged while moving the population eigenframe $\beta_1,\dots,\beta_k$
over all orthonormal frames of $\R^k$. The realized eigenframe $\omega_1,\dots,\omega_k$ does not move. 

\begin{proof}[Proof of {\bf Lemma~\ref{lem:orbit}}]
\emph{(i)} By definition \eqref{eq:SLam}, $\Sigma\in\mathcal S_\Lambda$ iff $\Sigma$ is positive definite and $M(\Sigma)$ has eigenvalues $\Lambda$.
Since symmetric matrices can be orthogonally diagonalized, the latter holds iff $M(\Sigma)=O^{\!\top}\Lambda O$ for
some $O\in \mathrm O(k)$.
Since $M(\Sigma)=G_B^{1/2}\Sigma G_B^{1/2}$ by definition,
this is equivalent to
$$\Sigma=G_B^{-1/2}\big(O^{\!\top}\Lambda O\big)G_B^{-1/2}=\Sigma_O.$$
We have shown that $\Sigma \in \mathcal{S}_\Lambda$ iff $\Sigma = \Sigma_O$ for some $O$, which establishes part (i).

\emph{(ii)} By definition, $M(\Sigma_O)=G_B^{1/2}\Sigma_OG_B^{1/2}=O^{\!\top}\Lambda O$, and
$O^{\!\top}\Lambda O\,(O^{\!\top}e_j)=O^{\!\top}\Lambda e_j=\mu_j\,O^{\!\top}e_j$. Since $\mu_j$ is a simple eigenvalue (Assumption~\ref{asm:sep}), the $j$th unit eigenvector of $M(\Sigma_O)$ is $\pm O^{\!\top}e_j$.
\end{proof}

\medskip

\begin{proof}[Proof of {\bf Theorem~\ref{thm:rotrange}}]
\emph{(i)}  The data $Y=BF+Z$ do not involve $\Sigma$, and $h_j$ is a function
of $Y$ alone; the $\lambda_j$ are the nonzero eigenvalues of
$F^{\!\top}G_BF/n$, which does not involve $\Sigma$ either.  By
\eqref{eq:thmfloor} the out-of-subspace error is determined by $n$, $\delta^2$
and $\lambda_j$, so it too is unchanged as $\Sigma$ varies.

\emph{(ii)}  By \eqref{eq:orbitlim} the members $\Sigma_O$ of $\mathcal S_\Lambda$ are
indexed by $O\in \mathrm{O}(k)$, and by \eqref{eq:Morbit} the corresponding population direction
is $\beta_j(\Sigma_O)=\pm O^{\!\top}e_j$.  As $O$ ranges over $\mathrm O(k)$, $O^{\!\top}e_j$ ranges over the full unit sphere of $\mathbb{R}^k$ (every unit vector is the $j$th column of some orthogonal matrix), while $\omega_j$, the $j$th eigenvector of the $\Sigma$-free matrix \eqref{eq:Mhatrecall}, is fixed. Hence, for $k \geq 2$,
$\sin^2\!\angle\big(\omega_j,\beta_j(\Sigma)\big)$
ranges over $[0,1]$.

\emph{(iii)}  The proof of Theorem~\ref{thm:main} uses the specific returns only through the
conclusions (a) and (b) of Proposition~\ref{prop:dual}, which hold for the
fixed $Z$ and do not involve $\Sigma$; hence \eqref{eq:thm} holds for the fixed
$(F,Z)$ and every admissible $\Sigma$.  By \eqref{eq:thm} the limiting estimation error is
$\alpha+(1-\alpha)\sin^2\!\angle(\omega_j,\beta_j)$, where
$\alpha=\delta^2/(n\lambda_j+\delta^2)\in(0,1)$ is held fixed by
part~(i), and therefore ranges over $[\alpha, 1]$
by part (ii).

\end{proof}

\medskip

\begin{proof}[Proof of {\bf Corollary~\ref{cor:noest}}]
Under each $\PP_\Sigma$, Assumptions~\ref{asm:fm}--\ref{asm:sep} hold, and
Assumption~\ref{asm:reg} holds almost surely because $F$ has a density.
Write $R(F,\Sigma):=\sin^2\!\angle\big(\omega_j,\beta_j(\Sigma)\big)$, where
$\omega_j=\omega_j(F)$.  By Theorem~\ref{thm:main}(ii),
$\sin^2\!\angle_p(\Pib h_j,b_j)\to R(F,\Sigma)$ $\PP_\Sigma$-almost surely.
Suppose a sequence $T_p$ as in the statement exists.  Then
\begin{equation*}
  T_p\longrightarrow R(F,\Sigma)\quad\text{in $\PP_\Sigma$-probability, for
  every admissible $\Sigma$.}
\end{equation*}

By Lemma~\ref{lem:orbit} and $k\ge2$ there are admissible $\Sigma,\Sigma'$
with $\beta:=\beta_j(\Sigma)$ and $\beta':=\beta_j(\Sigma')$ satisfying
$\beta'\ne\pm\beta$.  The laws of $F$ under $\PP_\Sigma$ and $\PP_{\Sigma'}$
are nondegenerate Gaussian laws on $\R^{k\times n}$, hence mutually absolutely
continuous, and the conditional law of $Z$ given $F$ is the same kernel $Q$
under both; therefore $\PP_\Sigma$ and $\PP_{\Sigma'}$ are mutually absolutely
continuous.  Convergence in probability is preserved under an absolutely
continuous change of measure, so $T_p\to R(F,\Sigma')$ in
$\PP_\Sigma$-probability as well.  Limits in probability are almost surely
unique, so $R(F,\Sigma)=R(F,\Sigma')$ for $\PP_\Sigma$-almost every $F$, that
is,
\[
  \langle\omega_j,\beta-\beta'\rangle\,\langle\omega_j,\beta+\beta'\rangle
  =\langle\omega_j,\beta\rangle^2-\langle\omega_j,\beta'\rangle^2=0
  \qquad\text{for Lebesgue-almost every }F .
\]

This is impossible, because the identity fails on a set of paths of positive
Lebesgue measure.  Write
$g(F):=\langle\omega_j(F),\beta\rangle^2-\langle\omega_j(F),\beta'\rangle^2$
for paths $F$ satisfying Assumption~\ref{asm:reg}; this does not depend on the
sign of $\omega_j(F)$.  Let
$F_0:=\sqrt n\,\Sigma^{1/2}\,[\,I_k\ \ 0\,]\in\R^{k\times n}$, a path whose
sample covariance $F_0F_0^{\!\top}/n$ equals $\Sigma$.  By \eqref{eq:Mdef},
$M_{F_0}=M(\Sigma)$, which has the simple positive eigenvalues
$\mu_1>\dots>\mu_k$; hence $F_0$ satisfies Assumption~\ref{asm:reg},
$\omega_j(F_0)=\pm\beta$, and
\[
  g(F_0)=1-\langle\beta,\beta'\rangle^2>0,
\]
since $\beta$ and $\beta'$ are unit vectors with $\beta'\ne\pm\beta$.  Now let
$F_m\to F_0$.  Then $M_{F_m}\to M_{F_0}$, and Lemma~\ref{lem:econv}(i), applied
to each of the $k$ simple eigenvalues of $M_{F_0}$, shows that for large $m$
the matrix $M_{F_m}$ has $k$ distinct positive eigenvalues, so that $F_m$
satisfies Assumption~\ref{asm:reg}, and that
$\angle\big(\omega_j(F_m),\beta\big)\to0$; by Lemma~\ref{lem:econv}(ii),
$\omega_j(F_m)\omega_j(F_m)^{\!\top}\to\beta\beta^{\!\top}$, and therefore
$g(F_m)\to g(F_0)$.  Consequently there is an open ball $U\subset\R^{k\times n}$
centered at $F_0$ on which Assumption~\ref{asm:reg} holds and $g>0$. 
The ball $U$ has positive Lebesgue measure, and $g\ne0$ on $U$.
\end{proof}
}

\section{Symbol Table} \label{sec:tables}

\begin{table}[htbp]
\centering
\footnotesize
\renewcommand{\arraystretch}{1.35}
\setlength{\tabcolsep}{5pt}
\begin{tabular}{@{}llccc@{}}
\hline
Matrix & Definition & Size & Eigenvalues & Eigenvectors \\
\hline
\multicolumn{5}{@{}l}{\emph{Ambient $p$-dimensional covariances}}\\
$\Sigma_0/p$ & $B\Sigma_f B^{\!\top}/p$ \ (rank $k$) & $p\times p$ & $\mu_j^{(p)}$ & $b_j$ \\
$S^{(p)}$ & $YY^{\!\top}/(np)$ & $p\times p$ & $\theta_j^{(p)}$ & $h_j=h_j^{(p)}$ \\
$S_0^{(p)}$ & $BFF^\top B^\top/(np) = \bar b\,M_F^{(p)}\,\bar b^{\!\top}$ \ (rank $k$, noiseless) & $p\times p$ & $\lambda_j^{(p)}$ & $h_j^0$ \\
\hline
\multicolumn{5}{@{}l}{\emph{$n\times n$ dual Grams }}\\
$W^{(p)}$ & $Y^{\!\top}Y/(np)$ & $n\times n$ & $\theta_j^{(p)}$ & $w_j^{(p)}$ \\
$W$ & $F^{\!\top}G_B F/n+(\delta^2/n)I_n$ \ ($=\lim_p W^{(p)}$) & $n\times n$ & $\lambda_j+\delta^2/n$ & $w_j$ \\
$W_0^{(p)}$ & $\Phi^{\!\top}\Phi/(np)=F^{\!\top}G_B^{(p)}F/n$ \ (noiseless dual) & $n\times n$ & --- & --- \\
$W_0$ & $F^{\!\top}G_B F/n$ \ (noiseless limit) & $n\times n$ & $\lambda_j$ & $w_j$ \\
\hline
\multicolumn{5}{@{}l}{\emph{$k\times k$ matrices in the loading frame $\bar b=B(B^{\!\top}B)^{-1/2}$}}\\
$M_F^{(p)}$ & $(G_B^{(p)})^{1/2}(FF^{\!\top}/n)(G_B^{(p)})^{1/2}=\bar b^{\!\top}S_0^{(p)}\bar b$ \ (dual of $W_0^{(p)}$) & $k\times k$ & $\lambda_j^{(p)}$ & $\omega_j^{(p)}$ \\
$M_F$ & $G_B^{1/2}(FF^{\!\top}/n)G_B^{1/2}$ \ ($=\lim_pM_F^{(p)}$; dual of $W_0$) & $k\times k$ & $\lambda_j$ & $\omega_j$ \\
$M_f^{(p)}$ & $(G_B^{(p)})^{1/2}\Sigma_f(G_B^{(p)})^{1/2}=\bar b^{\!\top}(\Sigma_0/p)\bar b$ & $k\times k$ & $\mu_j^{(p)}$ & $\beta_j^{(p)}=\bar b^{\!\top}b_j$ \\
$M_f$ & $G_B^{1/2}\Sigma_fG_B^{1/2}$ \ ($=\lim_pM_f^{(p)}$) & $k\times k$ & $\mu_j$ & $\beta_j$ \\
\hline
\multicolumn{5}{@{}l}{\emph{$k\times k$ loading Gram}}\\
$G_B^{(p)}$ & $B^{\!\top}B/p$ & $k\times k$ & $>0$ & --- \\
$G_B$ & $\lim_p B^{\!\top}B/p$ & $k\times k$ & $>0$ & --- \\
\hline
\end{tabular}
\caption{Important matrices,
with their eigenvalue and eigenvector
labels. Throughout $j=1,\dots,k$ indexes the top eigenpairs in decreasing
order. Gram duality (Lemma~\ref{lem:gramdual}) explains
the repeated eigenvalues across the different grouped blocks.
}

\label{tab:matrices}
\end{table}

\end{appendices}

\bibliography{references_2}

\end{document}